\documentclass{article} 
\usepackage[utf8]{inputenc}
\usepackage[T1]{fontenc}
\usepackage{indentfirst}
\usepackage{amsmath}
\usepackage{amssymb}
\usepackage{amsthm}
\usepackage{cancel}
\usepackage{dsfont}
\usepackage{stmaryrd}
\SetSymbolFont{stmry}{bold}{U}{stmry}{m}{n}
\usepackage{hyperref}
\usepackage{empheq}
\usepackage{cleveref}
\usepackage{bm}
\usepackage{enumitem}
\usepackage[pdftex]{graphicx}
\usepackage{xcolor}
\usepackage[toc,page]{appendix}

\title{Small noise limit and convexity for generalized incompressible flows, Schr\"odinger problems, and optimal transport}

\author{
Aymeric \textsc{Baradat}
\thanks{CMLS \'Ecole polytechnique, 91128 Palaiseau Cedex, FRANCE}
\thanks{DMA \'Ecole normale sup\'erieure, 45 rue d’Ulm, F-75230 Paris Cedex 05, FRANCE
\newline
\href{mailto:aymeric.baradat@polytechnique.edu}{\texttt{aymeric.baradat@polytechnique.edu}}
}
\and
L\'eonard \textsc{Monsaingeon}
\thanks{IECL Universit\'e de Lorraine, F-54506 Vandoeuvre-l\`es-Nancy Cedex, FRANCE
}
\thanks{GFM Universidade de Lisboa, Campo Grande, Edif\'icio C6, 1749-016 Lisboa, PORTUGAL
\newline
\href{mailto:leonard.monsaingeon@univ-lorraine.fr}{\texttt{leonard.monsaingeon@univ-lorraine.fr}}
}
}
\date{}

\theoremstyle{plain}
\newtheorem{Thm}{Theorem}[section]

\newtheorem{Prop}[Thm]{Proposition}
\newtheorem{Lem}[Thm]{Lemma}
\theoremstyle{definition}
\newtheorem{Def}[Thm]{Definition}

\theoremstyle{remark}
\newtheorem{Rem}[Thm]{Remark}

\newcommand{\R}{\mathbb{R}}
\newcommand{\Z}{\mathbb{Z}}
\newcommand{\T}{\mathbb{T}}
\newcommand{\E}{\mathbb{E}}
\newcommand{\N}{\mathbb{N}}
\newcommand{\A}{\mathcal{A}}
\newcommand{\F}{\mathcal{F}}
\newcommand{\Abar}{\overline{\mathcal{A}}}
\newcommand{\Hbar}{\overline{\mathcal{H}}}
\renewcommand{\AA}{\boldsymbol{\mathcal{A}}}
\newcommand{\FF}{\boldsymbol{\mathcal{F}}}
\newcommand{\HH}{\boldsymbol{\mathcal{H}}}
\newcommand{\PP}{{\boldsymbol{P}}}
\newcommand{\QQ}{\boldsymbol{Q}}
\newcommand{\rhont}{\rho^\nu_t}
\newcommand{\Tbar}{\overline{T}}
\newcommand{\Bbar}{\overline{B}{}}
\newcommand{\Rbar}{\overline{R}{}}
\newcommand{\narrowcv}{\overset{\ast}{\rightharpoonup}}
\newcommand{\XX}{\boldsymbol{X}}
\newcommand{\oomega}{\overline{\omega}}

\newcommand{\cg}{\langle}
\newcommand{\cd}{\rangle}

\newcommand{\OT}{\mathsf{OT}}
\newcommand{\Sch}{\mathsf{Sch}}
\newcommand{\REu}{\mathsf{REu}}
\newcommand{\Bro}{\mathsf{Br\ddot{o}}}
\newcommand{\MREu}{\mathsf{MREu}}
\newcommand{\MBro}{\mathsf{MBr\ddot{o}}}

\newcommand{\dist}{\mathsf{d}}
\newcommand{\dd}{\mathrm{d}}
\newcommand{\DMK}{\mathsf{d_{MK}}}
\newcommand\pf{_{\#}}
\newcommand{\Iiota}{\boldsymbol{\iota}}
\renewcommand{\P}{\mathcal{P}}
\renewcommand{\H}{\mathcal{H}}

\DeclareMathOperator{\Leb}{Leb}
\DeclareMathOperator{\Div}{div}

\DeclareMathOperator{\D}{d\!}

\DeclareMathOperator{\Id}{Id}
\DeclareMathOperator{\Inc}{\iota_{\small Inc}}
\DeclareMathOperator{\IInc}{\boldsymbol\iota_{\small Inc}}

\begin{document}

\title{Small noise limit and convexity for generalized incompressible flows, Schr\"odinger problems, and optimal transport
\thanks{L.M. was supported by the Portuguese Science Foundation through FCT project PTDC/MAT-STA/0975/2014 \textit{From Stochastic Geometric Mechanics to Mass Transportation Problems} and FCT project PTDC/MAT-STA/28812/2007 \textit{Schr\"oMoka}.
}
}
% \subtitle{Do you have a subtitle?\\ If so, write it here}

\title{Small noise limit and convexity for generalized incompressible flows, Schr\"odinger problems, and optimal transport}

\author{
Aymeric \textsc{Baradat}
\thanks{CMLS \'Ecole polytechnique, 91128 Palaiseau Cedex, FRANCE}
\thanks{DMA \'Ecole normale sup\'erieure, 45 rue d’Ulm, F-75230 Paris Cedex 05, FRANCE
\newline
\href{mailto:aymeric.baradat@polytechnique.edu}{\texttt{aymeric.baradat@polytechnique.edu}}
}
\and
L\'eonard \textsc{Monsaingeon}
\thanks{IECL Universit\'e de Lorraine, F-54506 Vandoeuvre-l\`es-Nancy Cedex, FRANCE
}
\thanks{GFM Universidade de Lisboa, Campo Grande, Edif\'icio C6, 1749-016 Lisboa, PORTUGAL
\newline
\href{mailto:leonard.monsaingeon@univ-lorraine.fr}{\texttt{leonard.monsaingeon@univ-lorraine.fr}}
}
}
\date{\today}

\maketitle

\begin{abstract}
This paper is concerned with six variational problems and their mutual connections: The quadratic Monge-Kantorovich optimal transport, the Schr\"odinger problem, Brenier's relaxed model for incompressible fluids, the so-called Br\"odinger problem recently introduced by M. Arnaudon \& \emph{al.} \cite{arnaudon2017entropic}, the multiphase Brenier model, and the multiphase Br\"odinger problem.
All of them involve the minimization of a kinetic action and/or a relative entropy of some path measures with respect to the reversible Brownian motion.
As the viscosity parameter $\nu\to 0$ we establish Gamma-convergence relations between the corresponding problems, and prove the convergence of the associated pressures arising from the incompressibility constraints.
We also present new results on the time-convexity of the entropy for some of the dynamical interpolations.
Along the way we extend previous results by H. Lavenant \cite{lavenant2017time} and J-D. Benamou \& \emph{al.} \cite{benamou2017generalized}.\\
\end{abstract}
\newpage
\tableofcontents

\section{Introduction}

 Since the works of Y. Brenier in the 90's \cite{brenier1989least,brenier1993dual,brenier1999minimal} it is known that there is a deep connection between generalized incompressible flows (inviscid Euler equations) and Monge-Kantorovich optimal transport.
 A decade later, C. L\'eonard, M. Cuturi, and T. Mikami established independently a link between deterministic optimal transport and the so-called Schr\"odinger problem (also referred to as \emph{entropic optimal transport}) \cite{mikami2004monge,leonard2012schrodinger,cuturi2013sinkhorn,leonard2013survey}.
 The two approaches are closely related, see \emph{e.g.} \cite{gentil2015analogy} for a thorough discussion.
 Recently both variational problems were blended into a single framework and the corresponding \emph{entropic interpolation problem for incompressible viscid fluids}, also known as the \emph{Br\"odinger} problem (contraction of Brenier and Schr\"odinger), was introduced in \cite{arnaudon2017entropic}.
 Just as Brenier's variational relaxed formulation leads to the incompressible Euler (inviscid) equation, this new approach leads to some particular incompressible motion of multiphase viscid fluids.
 See \cite{arnaudon2017entropic,baradat2018existence} for more discussions on these fluid-mechanical aspects.

In this paper, we will deal with six optimization problems: the Monge-Kantorovic optimal transport problem, the Schr\"odinger problem, Brenier's relaxed model for incompressible fluids, the Br\"odinger problem, and multiphase versions of the last two problems.
All of them amount to minimizing various path-measure action functionals, either an ensemble-average of the kinetic energy or a stochastic entropic version thereof, together with suitable constraints.
We will introduce these problems in full details in Section~\ref{sec:presentation}, but let us first present them in a few words:
In the optimal transport problem $\OT$ the action to be minimized is given by a deterministic average kinetic energy, and the initial and terminal distributions are prescribed.
The Schr\"odinger problem $\Sch_\nu$ is identical, except that an entropic stochastic perturbation is now added to the kinetic energy.
For Brenier's problem, \emph{i-e} the Relaxed Euler formulation $\REu$, the functional is again the total kinetic action, but the flow must now satisfy an additional incompressibility constraint for all times.
This forces in particular the initial and terminal distributions to equal the Lebesgue measure, and one is given instead the initial-to-terminal \emph{coupling} of the particle trajectories.
Br\"odinger's problem $\Bro_\nu$ is identical to Brenier's formulation, with the addition of an extra regularizing term as in $\Sch_\nu$.
Finally, the Multiphase Relaxed Euler and Multiphase Br\"odinger problems $\MREu,\MBro_\nu$ are multiphase versions of $\REu$ and $\Bro_\nu$, respectively, where the fluid under consideration is now seen as a superposition of a whole continuum of phases, each of them having prescribed initial and final distributions.

Let us highlight that the Schr\"odinger problem, the Br\"odinger problem and the multiphase Br\"odinger problem feature a viscosity parameter $\nu>0$ encoding the thermal, stochastic fluctuations in the system.
In this work we mainly take interest in the small noise limit $\nu\to 0$, and we present here a complete picture in terms of Gamma-convergence for the corresponding variational problems.
We fill some missing gaps in the theory, in particular we prove the convergence of the Br\"odinger problem towards Brenier's relaxed formulation (including multiphase versions thereof).

We will use slightly different approaches for $\REu$ and $\Bro_\nu$ on the one hand, and for $\OT$, $\Sch_{\nu}$, $\MREu$, and $\MBro_{\nu}$ on the other hand.
Indeed, for the latter four models we favor dynamical Benamou-Brenier-like formulations \cite{benamou2000computational}.
For the first two problems such a formulation does not exist, and we will argue directly at the level of the underlying stochastic processes and Lagrangian trajectories.
This will result in more involved proofs but also in finest statements, and explains why we will study first $\REu$ and $\Bro_\nu$ and then the other problems.

\paragraph{Informal presentation of the contents}
We will give rigorous statements of our main results later on in Section~\ref{section:contributions}, but for now let us briefly summarize our contributions and ideas.
As already mentioned, we will be mostly concerned with the small noise limit, and the convergences below should be understood as $\nu\to 0$.

Our first result (Theorem~\ref{thm:GammaCv_process}) will assert the $\Gamma$-convergence of the Br\"odinger problem towards Brenier's Relaxed formulation of Euler's equation, $\Bro_\nu\to\REu$.
This improves results of \cite{benamou2017generalized} where the same convergence was recently proved for a discrete-time version of the problem only, but our method of proof is completely different.
As already said, this will be the most involved part of the paper.
As often for Gamma-convergence problems, the $\Gamma-\liminf$ part will follow from more or less standard lower semicontinuity arguments and the main challenge is rather the construction of suitable recovery sequences for the $\Gamma-\limsup$.
Roughly speaking, the latter sequences will be constructed by stochastic deformation of classical trajectories: given a deterministic path-measure $P$, we will construct a stochastic perturbation $P^\nu\to P$ by superposing Brownian bridges of small variance $\nu$ centered at paths $\omega=(\omega_t)_{t\in [0,1]}$ initially charged by $P$.
A careful analysis will then reveal that the resulting Br\"odinger stochastic action of $P^\nu$ does not exceed the Brenier's deterministic action of $P$ by more than $\mathcal O(\nu)$.

Our second contribution consists in a new independent proof of convergence of the entropic Schr\"odinger problem towards deterministic optimal transport (Theorem~\ref{thm:CV_Sch_OT}), $\Sch_\nu\to\OT$, already known from \cite{mikami2004monge,leonard2012schrodinger,cuturi2013sinkhorn,carlier2017convergence}.
The main ingredient will be a quantified regularization procedure, Lemma~\ref{lem:fund}, which is nothing but an Eulerian PDE-version of our previous Lagrangian stochastic deformation:
Given a curve of probability measures $\rho=(\rho_t)_{t\in [0,1]}$ with finite kinetic energy, we construct a noisy approximation $\rho^\nu_t\coloneqq \tau_{\nu t(1-t)}*\rho_t$ by running the heat flow for short times ($\tau_s$ is the heat kernel at time $s>0$), and show that its dynamical Schr\"odinger action deviates from the Benamou-Brenier action of the original curve by at most $\mathcal O(\nu)$ in order to retrieve the $\Gamma-\limsup$.

Our third result (Theorem~\ref{thm:GammaCv_multiphase}) establishes the convergence of the Multiphase Br\"odinger problem towards the multiphase Relaxed Euler formulation, $\MBro_\nu\to\MREu$, and is completely new as far as we know.
The proof essentially consists in superposing by linearity our previous construction of recovery sequences for single phases, using the fact that the heat flow preserves the Lebesgue measure.
Since both problems include an incompressibility constraint some pressure fields arise as the corresponding Lagrange multipliers, and we prove convergence of the pressures $p^\nu\to p$ as well.

Finally, and as a byproduct of the previous analysis, our last set of contributions will focus on the time-convexity of the entropy functional along the relevant dynamical interpolations:
the optimal transport and the Schr\"odinger problems on the one hand (Proposition~\ref{prop:convexity_OT_Sch}), and on the other hand the multiphase Brenier  and multiphase Br\"odinger problems (Proposition~\ref{prop:convexity_MREu_MBro}).
Part of our results here are well-known \cite{mccann1997convexity,leonard2017convexity}, but we either provide again a new and rather elementary proof of independent interest, or improve recent results from \cite{lavenant2017time}.
The argument will be again purely variational: Starting from an optimal curve $\rho$, we exploit the same previous Eulerian regularization to construct a noisy competitor $\rho^\nu$ for the optimization problem under consideration.
The necessarily positive sign of the defect at first order $\mathcal O(\nu)$ reveals that the entropy of $\rho^\nu$ must be convex in time, and we conclude by taking $\nu\to 0$.
\paragraph{Outline.}
The paper is organized as follows: Section~\ref{sec:presentation} contains the precise formulations of the six variational problems, as well as rigorous statements of our main results and preliminary definitions.
In Section~\ref{section:GammaCv_process} we discuss the convergence $\Bro_\nu\to\REu$ by probabilistic arguments.
Section~\ref{section:convergence_Sch_OT} contains the PDE regularization procedure, Lemma~\ref{lem:fund}, as well as our new proof of the convergence $\Sch_\nu\to\OT$ of entropic towards deterministic optimal transport.
In Section~\ref{section:convergence_MBro_MREu} we prove the corresponding result for incompressible multiphase flows, namely $\MBro_\nu\to\MREu$.
We also show that the associated pressures converge.
Our last Section~\ref{section:convexity} is devoted to the time-convexity of the entropy in the various models.
We include in Appendix~\ref{sec:exist_unique_MBro} a self-contained proof of the existence and uniqueness of solutions for $\MBro$.
In Section~\ref{section:GammaCv_process} we shall heavily rely on some explicit properties of the Brownian motion and bridges on the torus: in order not to interrupt the exposition we will simply give the technical statements when needed in the text, and defer the proofs to Appendix~\ref{section:brownian_torus}.

%
%%%%%%%%%%%%%%%%%%%%%%%%%%%%%%%%%%%%%%%%%%%%%%%%%%%%%%%%%%%%%%%%%
\section{Formulation of the problems and main statements}
\label{sec:presentation}
In this paper we carry out the whole analysis in the $d$-dimensional flat torus $\T^d=(\R^d/\Z^d)$ for technical convenience, but some of the arguments might be adapted for convex domains $\Omega\subset\R^d$ (in which case the natural boundary conditions should be the homogeneous Neumann ones, corresponding to mass conservation and reflected Brownian motion).
We denote by $\dist$ the geodesic distance on the torus and $\pi: \R^d \to \T^d$ the canonical projection.
The normalized Lebesgue measure on the torus is denoted by $\Leb$ (sometimes also $\D x$ or $\D y$ when no ambiguity arises).
If $\mathcal{X}$ is a Polish space, $\P(\mathcal{X})$ stands for the set of Borel probability measures on $\mathcal{X}$.
If $\rho\in\P(\mathcal X)$ we use the notation $\D\rho$ in the integrals, i-e we write $\int_{\mathcal X}\varphi(x)\D\rho(x)$ for the integral of $\varphi$ with respect to the measure $\rho$.
When $\rho\in \P(\T^d)$ is absolutely continuous with respect to the Lebesgue measure $\D x$ we will often identify it with its Radon-Nikodym density, and simply write $\D\rho(x)=\rho(x)\D x$ with a slight abuse of notations.
The standard Brownian motion on the torus with diffusivity $\nu>0$ and starting from $x\in \T^d$ will be denoted by $R^\nu_x$, and similarly $\Rbar^\nu_{\bar x}$ will stand for the Brownian motion on the whole space started from $\bar x\in\R^d$.
The \emph{reversible Brownian motion} on the torus is obtained by choosing a uniform initial distribution,
$$
R^\nu\coloneqq \int_{\T^d}R^\nu_x\,\D x.
$$
\subsection{The variational problems}
First of all, let us describe precisely the six optimization problems that we will be dealing with throughout. 
\paragraph{The quadratic optimal transport problem.}
Given $\rho_0$ and $\rho_1$ in $\P(\T^d)$, the quadratic Monge-Kantorovich (square) distance is defined as:
\begin{equation}
\DMK^2(\rho_0, \rho_1) \coloneqq  \inf\limits_{\gamma}\,\frac{1}{2} \int \dist(x,y)^2 \D \gamma(x,y),
\label{eq:def_DMK}
\end{equation}
where the infimum runs over the set of \emph{admissible plans}, that are by definition the measures $\gamma \in \P(\T^d \times \T^d)$ with first marginal $\rho_0$ and second marginal $\rho_1$.
This provides a distance on $\P(\T^d)$ that metrizes the topology of narrow convergence\footnote{
We recall that narrow convergence is the weak-$*$ convergence of measures in duality with continuous, bounded test functions.
}, and we refer to the classical monographs \cite{villani2008optimal,santambrogio2015optimal} for a detailed account on the theory of optimal transport and extended bibliography.\\

Note that \eqref{eq:def_DMK} is a static formulation.
In order to introduce the celebrated dynamical Benamou-Brenier counterpart \cite{benamou2000computational}, let us first recall the notion of absolutely continuous curves.
Following \cite[Chapter 1]{ambrosio2006gradient}, if $(\mathcal X,\mathfrak d)$ is a Polish space and $1\leq p < +\infty$, we say that a curve $x:[0,T]\to \mathcal X$ belongs to $AC^p([0,T];\mathcal X)$ if the \emph{metric speed} $|\dot x|(t)\coloneqq \lim\limits_{s\to t}\frac{\mathfrak d(x(s),x(t))}{|s-t|}$ exists for a.e. $t\in [0,T]$ and belongs to $L^p(0,T)$.
In the (complete) Wasserstein space $(\P(\T^d),\DMK)$, $AC^2$ curves can be further characterized in terms of continuity equations and kinetic energy as:
\begin{Thm}[{\cite[Thm. 8.3.1]{ambrosio2006gradient}}]
\label{thm:characterization_AC2}
Let $\rho=(\rho_t)_{t\in[0,1]}$ be a curve from $[0,1]$ to $\P(\T^d)$.
Then $\rho$ belongs to $AC^2([0,1]; \P(\T^d))$ if and only if there exists a measurable vector field $c=c_t(x)$ belonging to $L^2([0,1]\times \T^d, \D t \otimes \rho_t)$,
\emph{i-e}
$$
\|c\|^2_{L^2([0,1]\times \T^d, \D t \otimes \rho_t)}\coloneqq \int_0^1 \hspace{-5pt} \int |c_t(x)|^2 \D \rho_t(x) \D t \quad <+\infty,
$$
such that the continuity equation
\begin{equation*}
\partial_t \rho_t + \Div(\rho_t c_t) = 0
\end{equation*}
holds in a weak sense.
In that case
\begin{equation}
\int_0^1 |\dot{\rho}_t|^2 \D t = \min\limits_c \int_0^1 \hspace{-5pt}
\int |c_t(x)|^2 \D \rho_t(x) \D t,
\label{eq:def_cmin_speed_MK}
\end{equation}
where the metric derivative $|\dot{\rho}_t|=\lim\limits_{h\to 0}\frac{\DMK(\rho_t,\rho_{t+h})}{h}$ is well defined for almost all $t$ and the minimum is taken among all such $L^2$ vector fields $c$.
\end{Thm}

This allows to reformulate the optimal transport problem by saying that a curve of measures $\rho=(\rho_t)_{t\in[0,1]}$ is admissible for the problem $\OT(\rho_0, \rho_1)$ if it is of regularity $AC^2([0,1]; \P(\T^d))$, and if it coincides with $\rho_0$ at time $t=0$ and with $\rho_1$ at time $t=1$.
It is a solution to $\OT(\rho_0, \rho_1)$ if it minimizes the kinetic action functional defined by
\begin{equation}
\label{eq:defA}
\A(\rho) \coloneqq  \frac{1}{2} \int_0^1 |\dot{\rho}_t|^2 \D t
\end{equation}
for absolutely continuous curves, and set to $+ \infty$ otherwise, on the set of all admissible curves.

The equivalence between the static formulation and the dynamical one is due to Benamou and Brenier in \cite{benamou2000computational}, and this formulation in terms of metric derivatives can be found for example in \cite[Chapter 5]{santambrogio2015optimal}.
The functional $\A$ is convex, proper (the preimage of bounded sets are relatively compact) and lower semicontinuous for the topology of uniform convergence on $C([0,1] ; \P(\T^d))$ when $\P(\T^d)$ is endowed with the distance $\DMK$.
There always exists at least one minimizer -- as for the static problem --  but uniqueness does not hold in general (typically due to non-uniqueness of geodesics in the torus).

\paragraph{The Schr\"odinger problem.}
We also formulate this problem in a dynamical Benamou-Brenier version.
To do so, we first need to introduce another functional on the curves of probability measures:
if $\rho$ is a curve on $\P(\T^d)$ with the additional regularity $\sqrt{\rho} \in L^2([0,1]; H^1(\T^d))$ (recall that when $\rho$ has density with respect to the Lebesgue measure we abuse notations and write $\D\rho_t(x)=\rho_t(x)\D x$), the following vector field is well defined for almost all $t\in (0,1)$ and $\rho_t$-almost all $x\in\T^d$:
\begin{equation*}
\nabla \log \rho_t \coloneqq  2 \frac{\nabla \sqrt{\rho_t}}{\sqrt{\rho_t}}.
\end{equation*}
For such curves $\rho$, we define
\begin{equation}
\label{eq:defF}
\F(\rho) \coloneqq  \frac{1}{8} \int_0^1 \hspace{-5pt}\int | \nabla \log \rho_t |^2 \D \rho_t \D t = 
\frac{1}{2}\int_0^1 \hspace{-5pt}\int |\nabla \sqrt{\rho_t}|^2 \D x \D t.
\end{equation}
If $\rho$ is not regular enough we set $\F(\rho) \coloneqq  + \infty$.
The quantity $\F(\rho)$ is nothing but the time-integral of the \emph{Fisher information}
$$
F(\rho_t)\coloneqq \frac 12\int \left|\nabla\frac 12\log\rho_t\right|^2\D \rho_t
$$
along the curve $\rho$.
(We use here the customary probabilistic $\frac 12$ factor for the generator of the Brownian motion, accounting for the final $\frac 18$ factor in \eqref{eq:defF}: indeed the osmotic velocity field in $\frac{1}{2}\Delta\rho=\Div(\rho\nabla\frac{1}{2}\log\rho)$ is $w(\rho)=\nabla\frac 12\log\rho$, and $F(\rho)=\frac 12\int |w(\rho)|^2\D\rho$ is the corresponding kinetic energy.)
It is strictly convex and lower semicontinuous for the topology of uniform convergence on $C([0,1] ; \P(\T^d))$.

Given $\rho_0,\rho_1$ in $\P(\T^d)$ and a diffusion parameter $\nu>0$, we say that a curve $\rho$ is admissible for the Schr\"odinger problem $\Sch_{\nu}(\rho_0, \rho_1)$ if it belongs to $AC^2([0,1]; \P(\T^d))$, if it coincides with $\rho_0$ at time $0$ and with $\rho_1$ at time $1$, and if $\F(\rho) < + \infty$.
A solution to $\Sch_{\nu}(\rho_0, \rho_1)$ is an admissible curve that minimizes
\begin{equation}
\label{eq:defH}
\H_{\nu}(\rho) \coloneqq  \frac{1}{2} \int_0^1 |\dot{\rho}_t|^2 \D t + \frac{\nu^2}{8} \int_0^1 \hspace{-5pt}\int | \nabla \log \rho_t |^2 \D \rho_t \D t =\A(\rho) + \nu^2 \mathcal{F}(\rho)
\end{equation}
in the set of all admissible curves. 

The Schr\"odinger problem dates back to articles by E. Schr\"odinger himself \cite{schrodinger1931umkehrung,schrodinger1932theorie}. Since then, it has given rise to a large amount of work, see for example \cite{zambrini1986variational,follmer1988random,cruzeiro2000bernstein,leonard2013survey} and references therein.
Classically, this problem is rather formulated in terms of relative entropy with respect to the Brownian motion, as we will do for the Br\"odinger problem.
Our dynamical framework is however equivalent, as observed in numerous papers -- see for instance \cite[Section~IV]{chen2016relation}, \cite[Cor.~5.8]{gentil2015analogy} or the introduction of \cite{gigli2018benamou}.

The functional $\H_{\nu}$ is strictly convex, proper, and lower semicontinuous for the topology of uniform convergence on $C([0,1] ; \P(\T^d))$.
As proved for example in \cite{leonard2013survey}, minimizers exist if and only if the initial and final entropies are finite:
\begin{equation*}
\int \rho_0 \log \rho_0 \,\D x< + \infty \qquad \mbox{and} \qquad \int \rho_1 \log \rho_1 \,\D x< + \infty,
\end{equation*}
and uniqueness is automatically guaranteed by strict convexity.

In most situations the structure of the minimizers is well understood, we refer once again to \cite{leonard2013survey} for a survey on the following results:
Typically (and usually requiring soft conditions on $\rho_0$ and $\rho_1$), the curve $\rho$ is a solution to $\Sch_\nu(\rho_0,\rho_1)$ if and only if there are two nonnegative functions $f = f(x)$ and $g = g(x)$ such that for all $t \in [0,1]$
\begin{equation*}
\rho_t = (f * \tau_{\nu t}) \times (g * \tau_{\nu (1-t)}),
\end{equation*}
where $(\tau_s)_{s \geq 0}$ is the heat kernel on $\T^d$.
Due to the endpoint constraints, $f$ and $g$ are obtained by solving 
the so-called \emph{Schr\"odinger} system:
\begin{equation*}
\left\{
\begin{aligned}
f \times (g * \tau_{\nu}) &= \rho_0,\\
(f * \tau_\nu )\times g &= \rho_1,
\end{aligned}
\right.
\end{equation*}
a nonlinear integral system studied in great details by various authors \cite{fortet1940resolution,beurling1960automorphism,jamison1975markov}.
Most of the known results on the Schr\"odinger problem strongly leverage this decomposition property of $\rho$.
We stress that our variational analysis below will not exploit this structure at all, hence our approach is more robust when new constraints are added.

The Schr\"odinger problem is used nowadays as an entropic regularization of optimal transport.
Indeed, the functional $\H_{\nu}$ Gamma-converges to the functional $\A$, as proved by L\'eonard in \cite{leonard2012schrodinger} and by Carlier \emph{et Al.} in \cite{carlier2017convergence} after partial results by Mikami in \cite{mikami2004monge}. (We also provide a new independent proof of this fact, see Theorem~\ref{thm:CV_Sch_OT} below.)
One of the advantages of this entropic optimal transport is clearly that $\H_\nu$ has a unique minimizer, but most importantly that it is more tractable numerically speaking and allows for extremely fast and efficient computational strategies for optimal transport -- see e.g. \cite{cuturi2013sinkhorn,benamou2015iterative}.

\paragraph{The Brenier model for incompressible fluids.}
Brenier's original model \cite{brenier1989least} is somehow an incompressible version of the optimal transport problem.
Here the data is a plan $\gamma \in \P(\T^d \times \T^d)$ whose marginals are both the Lebesgue measure (we say that $\gamma$ is \emph{bistochastic}), corresponding to assigning the initial-to-terminal distribution of particles in the fluid.
Brenier's problem is an optimization problem in the set of \emph{generalized flows}, \emph{i-e} in the set of probability measures on continuous paths on the torus $\P(C([0,1]; \T^d))$
(This is the set of laws of continuous processes.)
In this paper we choose to adopt the notation $\REu$ for ``Relaxed Euler'' instead of $\mathsf{Bre}$ (for ``Brenier'') in order to avoid confusion with our notation $\Bro$ for the Br\"odinger problem.
A generalized flow $P\in \P(C([0,1]; \T^d))$ is said to be admissible for Brenier's problem $\REu(\gamma)$ if:
\begin{enumerate}[label=(\roman*)]
\item
\label{item:joint_law}
the joint law $P_{0,1}$ between the initial and final times under $P$ is $\gamma$, which means that for all test functions $\varphi$ on $\T^d \times \T^d$,
\begin{equation}
\label{eq:gamma_endpoint}
\int \varphi(\omega_0, \omega_1) \D P(\omega) = \int \varphi(x,y) \D \gamma(x,y),
\end{equation}
\item
\label{item:incompressibility}
$P$ is incompressible, \emph{i-e} the marginal $P_t$ of $P$ at time $t$ is the Lebesgue measure: For all test-functions $\varphi$ on $\T^d$ and all $t \in [0,1]$,
\begin{equation}
\label{eq:incompressible_flow}
\int \varphi(\omega_t) \D P(\omega) = \int \varphi(x) \D x,
\end{equation}
\item 
\label{item:finite_action}
the following action functional is finite:
\begin{equation*}
\Abar(P) \coloneqq  \int A(\omega) \D P(\omega) < + \infty,
\end{equation*}
where the action $A$ of a path $\omega\in C([0,1];\T^d)$ is defined by
\begin{equation}
A(\omega) \coloneqq  \frac{1}{2} \int_0^1 |\dot{\omega}_t|^2 \D t
\label{eq:def_kinetic_action_path}
\end{equation}
for absolutely continuous curves, and is otherwise set to $+\infty$.
In particular an admissible $P$ should only charge absolutely continuous (in fact $AC^2$) curves.
\end{enumerate}
A solution to $\REu(\gamma)$ is then an admissible generalized flow which minimizes $\Abar$ in the set of all admissible generalized flows.

This model has been introduced and studied for the first time by Y. Brenier in \cite{brenier1989least} as a relaxation of the incompressible Euler equation in its variational formulation (due to V. Arnold in \cite{arnold1966geometrie}, developed in \cite{arnold1999topological}).
The functional $\Abar$ is affine, proper and lower semicontinuous for the topology of narrow convergence, and there always exist minimizers \cite[Section~4]{brenier1989least}.
Note that uniqueness does not hold in general, see \emph{e.g.} \cite{bernot2009generalized}.

\paragraph{The Br\"odinger problem.}
The Br\"odinger problem is an entropic version of the Brenier problem.
Given a bistochastic $\gamma\in \P(\T^d\times\T^d)$ and a diffusivity parameter $\nu>0$, a generalized flow $P$ is said to be admissible for the problem $\Bro_\nu(\gamma)$ if conditions \ref{item:joint_law} and \ref{item:incompressibility} in Brenier's model hold, and if \ref{item:finite_action} is replaced by:
\begin{enumerate}[label=(iii)']
\item the relative entropy of $P$ with respect to the reversible Brownian motion $R^{\nu}$ with diffusivity $\nu$ is finite,
\begin{equation*}
\Hbar_{\nu}(P) \coloneqq  \nu H(P \,|\, R^{\nu}) < + \infty.
\end{equation*}
\end{enumerate}
The relative entropy functional $(p,r) \mapsto H(p\, |\, r)$ will be properly defined in section \ref{sec:notations} below.
A solution to $\Bro_\nu(\gamma)$ is then an admissible generalized flow $P$ that minimizes $\Hbar_{\nu}$ in the set of all admissible generalized flows.

The Br\"odinger problem has been introduced and studied for the fist time in \cite{arnaudon2017entropic}.
The relative entropy $\H_\nu$ is strictly convex, proper and lower semicontinuous for the topology of narrow convergence, and it is proved in \cite{arnaudon2017entropic} that there exists a unique minimizer if and only if the following marginal entropy is finite:
\begin{equation*}
H(\gamma \,|\, \Leb \otimes \Leb) < + \infty.
\end{equation*}
As in the Schr\"odinger problem, uniqueness is guaranteed by strict convexity of the relative entropy.

\paragraph{The multiphase Brenier model.}
The Brenier and Br\"odinger models each have a multiphase version.
These are useful when studying the pressure field (which is the Lagrange multiplier associated with the incompressibility constraint) in these generalized incompressible models.
For Brenier's model, the multiphase version was introduced by Brenier himself in \cite{brenier1993dual,brenier1997homogenized}.
Following Lavenant in \cite{lavenant2017time}, we adopt here the point of view of \emph{traffic plans} (see also \cite{bernot2005traffic}), which are probability measures on the set $C([0,1]; \P(\T^d))$, \emph{i-e} elements of the set $\P(C([0,1]; \P(\T^d)))$.
Heuristically, each curve of measures $\rho\in C([0,1]; \P(\T^d))$ charged by a traffic plan $\PP\in \P(C([0,1]; \P(\T^d)))$ represents a phase in the fluid.
The ideal fluid can then be seen as a (possibly continuous) superposition of all those phases.
All the phases are coupled by an incompressibility constraint, and evolve so as to minimize the total kinetic action.

In this setting, the data is a probability measure $\Gamma \in \P\left( \P(\T^d) \times \P(\T^d) \right)$ with the compatibility condition
\begin{equation*}
\int \rho_0 \D  \Gamma(\rho_0, \rho_1) = \int \rho_1 \D  \Gamma(\rho_0, \rho_1) = \Leb.
\end{equation*}
We say that $\Gamma$ is \emph{bistochastic in average}.
Intuitively, $\D \Gamma(\rho_0, \rho_1)$ can be thought of as the fraction of phases $\rho$ that coincide with $\rho_0$ at time $0$ and with $\rho_1$ at time $1$.
A traffic plan $\PP$ is declared admissible for the problem $\MREu(\Gamma)$ if:
\begin{enumerate}[label=(\roman*)]
\item \label{item:Mjoint_law} the joint law between the initial and terminal times under $\PP$ is $\Gamma$, which means that for all test function $\varphi$ on $\P(\T^d) \times \P(\T^d)$,
\begin{equation}
\label{eq:Mjoint_law}
 \int \varphi(\rho_0, \rho_1)  \D \PP(\rho) = \int \varphi(\rho_0,\rho_1) \D \Gamma(\rho_0,\rho_1),
\end{equation}
\item
\label{item:Mincompressibility} 
for all $t$, the average of the density at time $t$ under $\PP$ is the Lebesgue measure, which means that for all $t \in [0,1]$,
\begin{equation}
\label{eq:incomressibility_MREu}
\int \rho_t \D \PP(\rho) = \Leb,
\end{equation}
\item
\label{item:Mfinite_action}
the following action functional is finite:
\begin{equation}
\label{eq:Mfinite_action}
\AA(\PP) \coloneqq  \int \A(\rho) \D \PP(\rho) < + \infty,
\end{equation}
where $\A$ is defined in \eqref{eq:defA}.
In particular $\PP$ should only charges absolutely continuous (and in fact $AC^2$) curves.
\end{enumerate}
A solution to $\MREu(\Gamma)$ is then a minimizer of $\AA$ over all admissible traffic plans.
 
The functional $\AA$ is affine, proper and lower semicontinuous for the topology of narrow convergence, and there always exist minimizers as proved in \cite{brenier1993dual}.
(The proof only requires slight modifications in order to fit within our description in terms of traffic plans, see also \cite[Thm. 2.12]{lavenant2017time})
Analogously to the Brenier model, uniqueness does not hold in general.
 
 Brenier's original model can be seen as a particular case of the multiphase Brenier model.
 Indeed, if $\gamma\in\P(\T^d\times\T^d)$ is bistochastic, one can canonically construct an associated $\Gamma\in \P\left( \P(\T^d) \times \P(\T^d) \right)$ by prescribing, for all test functions $\varphi$ on $\P(\T^d) \times \P(\T^d)$:
 \begin{equation}
 \label{eq:gamma_Gamma}
\int \varphi(\rho_0, \rho_1) \D \Gamma(\rho_0, \rho_1) = \int \varphi(\delta_x, \delta_y) \D \gamma (x,y).
 \end{equation}
One can then slightly adapt \cite[Section~4]{ambrosio2009geodesics} and build a solution to $\MREu(\Gamma)$ from a solution to $\REu(\gamma)$, and \emph{vice versa}.
We also refer to \cite{lavenant2017time} for an explanation on how to reformulate the works of Brenier and Ambrosio--Figalli in terms of traffic plans. 
%
%%%%%%%%%%%%%%%%%%%%%%%%%%%%%%%%%%%%
\paragraph{The multiphase Br\"odinger problem.}
The multiphase Br\"odinger problem $\MBro_\nu$ is to the multiphase Brenier model what the Schr\"odinger problem is to the optimal transport problem, namely an entropic stochastic regularization of some sort.
As for the multiphase Brenier model, we choose here an exposition in terms of traffic plans.
If $\Gamma$ is bistochastic in average and $\nu>0$, a traffic plan $\PP\in \P(C([0,1];\P(\T^d)))$ is said to be admissible for $\MBro_{\nu}(\Gamma)$ if points \ref{item:Mjoint_law} \ref{item:Mincompressibility} \ref{item:Mfinite_action} in the multiphase Brenier model hold, and if in addition:
\begin{enumerate}[label=(iv)]
\item the \emph{average Fisher information}
\begin{equation}
\label{eq:Mfinite_fischer}
\FF(\PP) \coloneqq  \int \F(\rho) \D \PP(\rho)
\end{equation}
is finite. We recall that $\F$ is defined in \eqref{eq:defF}.
\end{enumerate}
A solution to $\MBro_{\nu}(\Gamma)$ is then a traffic plan that minimizes the functional
\begin{equation*}
\HH_{\nu}(\PP) \coloneqq  \AA(\PP) + \nu^2 \FF(\PP) = \int \H_{\nu}(\rho) \D \PP(\rho)
\end{equation*}
 among all admissible traffic plans (where $\H_{\nu}$ is defined in \eqref{eq:defH}). 
In this setting, the problem was introduced by the first author in \cite{baradat2018existence} in order to prove the existence of a scalar pressure field in the Br\"odinger problem.
The functional $\HH_{\nu}$ is affine, proper and lower semicontinuous for the topology of narrow convergence.
In appendix~\ref{sec:exist_unique_MBro} we prove existence and uniqueness of minimizers under the entropy condition:
\begin{equation}
\label{eq:M_initial_final_entropy_finite}
\int \hspace{-5pt} \int \rho_0 \log \rho_0 \D \Gamma(\rho_0, \rho_1) < + \infty \quad \mbox{and} \quad \int \hspace{-5pt} \int \rho_1 \log \rho_1 \D \Gamma(\rho_0, \rho_1) < + \infty.
\end{equation}
In fact this condition is also necessary for the existence of minimizers because any plan with $\HH_\nu(\PP)<\infty$ must have finite marginal entropies, but we omit the details (see Remark~\ref{rem:necessary_marginal_entropies} later on).
However, contrarily to what happens in the Brenier model, the Br\"odinger problem is not  a particular case of the multiphase Br\"odinger problem.
Indeed, the construction \eqref{eq:gamma_Gamma} described in the Brenier case gives rise to a $\Gamma$ that cannot satisfy the entropy condition \eqref{eq:M_initial_final_entropy_finite} since $H(\delta_x\,|\,\Leb)=+\infty$ for all $x\in\T^d$.
There is still a link between $\Bro$ and $\MBro$, but we refrain from discussing further this connection and rather refer to \cite[Section~9]{baradat2018existence}.
\subsection{The pressure field in the incompressible models}
As already mentioned, each of the four incompressible models ($\REu$, $\MREu$, $\Bro$, $\MBro$) features its own Lagrange multiplier corresponding to the incompressibility constraint.
This \emph{pressure field}, a scalar distribution, can be interpreted as the differential of the optimal action when the prescribed density is perturbed, in accordance with the envelope theorem.
In this paper we will study these pressure fields for the multiphase models only, but similar results might be developed in the Brenier model or for the Br\"odinger problem. 

The precise definition of the pressure fields will be recalled later on in Theorem~\ref{thm:existence_pressure}, but let us make a bit more explicit the idea.
One can prove that, if $\PP$ is a solution to $\MREu(\Gamma)$, then there exists a unique $p \in \mathcal{D}'((0,1)\times \T^d)$ such that, for all $\QQ$ satisfying the joint law condition~\eqref{eq:Mjoint_law} but with perturbed density
\begin{equation*}
\int \rho_t \D \QQ(\rho) = (1 + \varphi) \Leb
\end{equation*}
for small but arbitrary $\varphi \in \mathcal{D}((0,1) \times \T^d)$, there holds
\begin{equation*}
\AA(\QQ) \geq \AA(\PP) + \cg p, \varphi \cd.
\end{equation*}
As a consequence, in the perturbed problem with perturbed density $(1 + \varphi) \Leb$ instead of $\Leb$, the optimal action is perturbed at first order by the quantity $\cg p, \varphi\cd$. 
This was first proved in \cite{brenier1993dual} for the Brenier model (the proof in the multiphase Brenier model case is a direct adaption).
As proved by the first author in \cite{baradat2018existence}, it is still true replacing $\MREu(\Gamma)$ by $\MBro_\nu(\Gamma)$ and $\AA$ by $\HH_\nu$.

In turn, the pressure field is a relevant object to study the dynamics of the solutions: It is a self-induced potential accelerating the particles (i-e $\ddot\omega_t=-\nabla p(t,\omega_t)$ for almost all trajectories at least when the pressure is regular enough) in such a way as to preserve the incompressibility constraint.
We refer to \cite[Thm. 6.8]{ambrosio2009geodesics} for an illustration of this fact for the Brenier model $\REu$.
%
%%%%%%%%%%%%%%%%%%%%%%%%%%%%%%%%%%%%%%%%%%%%%%%%%%%%%
\subsection{Contributions}
\label{section:contributions}
Our first result will assert the convergence of the Br\"odinger problem towards the Brenier model as the diffusivity $\nu\to 0$.
If $\gamma \in \P(\T^d \times \T^d)$ we write $\iota_{\gamma}$ for the characteristic function corresponding to the marginal constraint \eqref{eq:gamma_endpoint}, and $\Inc$ is the characteristic function corresponding to the incompressibility constraint \eqref{eq:incompressible_flow}.
In other words, if $P$ is a generalized flow,
\begin{equation*}
\iota_{\gamma}(P) \coloneqq  \left\{\begin{aligned}
&0 & \mbox{if }P_{0,1}=\gamma\\
&+ \infty & \mbox{else,}
\end{aligned} \right.
\quad \mbox{and} \quad 
\Inc(P) \coloneqq  \left\{ \begin{aligned}
&0 & \mbox{if }P_t=\Leb \mbox{ for all }t\\
&+ \infty & \mbox{else.}
\end{aligned} \right.
\end{equation*}
Our statement reads then
\begin{Thm}
\label{thm:GammaCv_process}
With the same notations as before,
	\begin{enumerate}
		\item
		\label{item:GammaCv_no_marge}
		The following $\Gamma$-convergence holds:
		\begin{equation*}
		\Gamma-\lim_{\nu \to 0}\,\{ \Hbar_{\nu} + \Inc \}= \Abar + \Inc.
		\end{equation*} 
		\item \label{item:GammaCv_with_marge} If $\gamma$ is bistochastic and satisfies $H(\gamma \, | \, \Leb \otimes \Leb) < + \infty$, then
		\begin{equation*}
		\Gamma-\lim_{\nu \to 0} \{\Hbar_{\nu} + \Inc + \iota_{\gamma}\} = \Abar + \Inc + \iota_{\gamma}.
		\end{equation*}
	\end{enumerate}
\end{Thm} 
This will be seen as a straightforward consequence of Theorem \ref{thm:general_GammaCv_process} below, and the result is stronger than convergence of the minimizers.
Note that the second part of our statement only addresses the case of a fixed marginal law $P_{0,1}=\gamma\in\P(\T^d\times\T^d)$, while the first part does not and will typically require suitable regularization $\gamma^\nu\narrowcv\gamma$.
The key step in the proof will be to build a recovery sequence $P^\nu$ by adding a Brownian bridge (with diffusivity $\nu$) to any absolutely continuous curve charged by any admissible generalized flow $P$ in the Brenier model.
We will then use a Cameron-Martin formula to compute the entropy of the resulting process.
We will also exploit the narrow continuity of the optimal action in Brenier's model $\REu(\gamma)$ with respect to $\gamma$ \cite[Thm.~1]{baradat2018continuous} to obtain a necessary and sufficient condition for a sequence $(\gamma^\nu)_{\nu>0}$ of bistochastic measures to be the marginal laws of a recovery sequence $P^\nu\narrowcv P$, see point~\ref{item:general_Gamma-limsup} in Theorem~\ref{thm:general_GammaCv_process} later on.

\bigskip
We will address next the convergence $\Sch_\nu\to\OT$ of entropic towards deterministic optimal transport:
For curves $\rho\in C([0,1];\P(\T^d))$, and given $\rho_0,\rho_1\in \P(\T^d)$, we write
\begin{equation*}
 \iota_{(\rho_0,\rho_1)}(\rho)=
\left\{\begin{aligned}
&0 & \mbox{if }\rho|_{t=0}=\rho_0 \mbox{ and } \rho|_{t=1}=\rho_1\\
&+ \infty & \mbox{else,}
\end{aligned}\right.
\end{equation*}
for the characteristic function of the endpoints constraint.
The entropy functional will be rigorously defined in section~\ref{sec:notations}, but for now let us anticipate and write $H(\mu)\coloneqq H(\mu\,|\,\Leb)$ for the entropy of $\mu\in\P(\T^d)$ computed relatively to the Lebesgue measure on the torus.
The statement then reads
\begin{Thm}
\label{thm:CV_Sch_OT}
Let $\rho_0,\rho_1\in\P(\T^d)$ such that $H(\rho_0)<+\infty$ and $H(\rho_1)<+\infty$.
Then
\begin{equation}
\label{eq:Gamma_lim_Sch_OT}
 \Gamma-\lim\limits_{\nu\to 0}\,\Big\{ \H_\nu + \iota_{(\rho_0,\rho_1)} \Big\}
 = \A + \iota_{(\rho_0,\rho_1)}
\end{equation}
for the uniform topology on $C([0,1];\P(\T^d))$.
\end{Thm}
We stress again that the result is not new \cite{mikami2004monge,leonard2012schrodinger,cuturi2013sinkhorn,carlier2017convergence}, but we will give an independent proof that is elementary and new to the best of our knowledge.
In particular, we will provide an explicit PDE construction of recovery sequences inspired from the previous probabilistic arguments and Brownian bridges.
% This procedure provides directly a recovery sequence. 
\bigskip

Our third result is the convergence $\MBro_{\nu}\to\MREu$ of the multiphase Br\"odinger problem  towards the multiphase Brenier model:
\begin{Thm}
	\label{thm:GammaCv_multiphase}
If $\Gamma$ is bistochastic in average with finite average marginal entropies as in \eqref{eq:M_initial_final_entropy_finite}, then
		\begin{equation}
		\label{eq:Gamma_lim_MBro_MREu}
		\Gamma-\lim_{\nu \to 0} \,\Big\{\HH_{\nu} + \IInc + \Iiota_{\Gamma}\Big\} = \AA + \IInc + \Iiota_{\Gamma}
		\end{equation}
for the narrow topology of $\P(C([0,1];\P(\T^d)))$.
\end{Thm} 
Here we write as before $\Iiota_{\Gamma}(\PP)$ and $\IInc(\PP)$ for the characteristic functions of the marginal and incompressibility constraints, \eqref{eq:Mjoint_law} and \eqref{eq:incomressibility_MREu}, respectively.
Again, this result is stronger than convergence of the minimizers.

Given that both problems include an incompressibility constraint, a natural question to ask is whether the associated Lagrange multipliers converge as well, \emph{i-e} whether the Br\"odinger pressure $p^\nu$ converges towards the Brenier pressure $p$.
The answer to that question is positive, but in order to make a rigorous statement we first need to introduce the following functional space:
\begin{Def}
\label{defi:functional_space_G}
	Let $C^{1, 2}([0,1] \times \T^d)$ denote the space of functions having continuous time-derivative and continuous second order space-derivatives.
	For such a function $f \in C^{1, 2}([0,1] \times \T^d)$, we say that $f \in \mathcal{G}$ if in addition:
	\begin{itemize}
		\item for all $x \in \T^d$, $f(0,x) = f(1,x) = 0$,
		\item for all $t \in [0,1]$, 
		\begin{equation*}
		\int_{\T^d} f(t,x) \D x = 0.
		\end{equation*}
	\end{itemize}
	Endowed with its natural $C^{1,2}$ norm, $\mathcal G$ is a Banach space.
\end{Def}
Our result is the following:
\begin{Thm}
\label{thm:CV_pressures}
Take $\Gamma$ bistochastic in average and satisfying \eqref{eq:M_initial_final_entropy_finite}.
For all $\nu>0$ let $p^{\nu}$ be the pressure field associated with $\MBro_{\nu}(\Gamma)$, and let $p$ be the pressure field associated with $\MREu(\Gamma)$ (both being defined in Theorem~\ref{thm:existence_pressure}).
Then
\begin{equation*}
p^{\nu} \underset{\nu \to 0}{\narrowcv} p \quad \mbox{in }\mathcal{G}'
\end{equation*}
for the weak-$\ast$ convergence on the topological dual $\mathcal{G}'$ of $\mathcal{G}$.
\end{Thm}
\bigskip

Finally, we will investigate the time-convexity of the entropy $H=H(\bullet\,|\,\Leb)$ along solutions to (some of) the previous problems.
Note that, for the one-phase problems $\REu,\Bro$, the incompressibility constraint $\rho_t=\Leb$ forces the entropy to be constant in time $H(\rho_t)\equiv H(\Leb)=0$, so nothing interesting can be said there.
We will give a new independent proof of the convexity along the interpolations $\OT,\Sch_\alpha$:
\begin{Prop}
 \label{prop:convexity_OT_Sch}
 Let $\rho_0,\rho_1\in\P(\T^d)$ have finite entropies $H(\rho_0),H(\rho_1)<\infty$, and let $\rho$ be a solution of $\OT(\rho_0,\rho_1)$ or $\Sch_\alpha(\rho_0,\rho_1)$ for fixed diffusivity $\alpha>0$.
 Then $t\mapsto H(\rho_t)$ is convex.
\end{Prop}
For the optimal transport problem this is nothing but R.~McCann's celebrated \emph{displacement convexity} \cite{mccann1997convexity}, and the convexity for the Schr\"odinger problem is also known from \cite{leonard2017convexity}.
However, we would like to stress that our unified proof is elementary, purely variational, and exploits neither prior knowledge on - nor particular structure of - the minimizers.

In the multiphase settings $\MREu,\MBro_\alpha$ we will establish
\begin{Prop}
\label{prop:convexity_MREu_MBro}
Let $\Gamma$ be bistochastic in average with finite marginal entropies as in \eqref{eq:M_initial_final_entropy_finite}, and let $\PP$ be \emph{any} solution to $\MREu(\Gamma)$ or $\MBro_\alpha(\Gamma)$ for $\alpha>0$.
Then the average entropy
\begin{equation*}
 t\mapsto \int H(\rho_t)\D \PP(\rho)
\end{equation*}
is convex.
\end{Prop}
The convexity for $\MREu$ was conjectured by Y. Brenier in \cite{brenier2003extended}.
This was recently proved by H. Lavenant in \cite{lavenant2017time} for particular solutions only (roughly speaking, solutions with minimal entropy in some integral sense), while we stress that our statement holds for \emph{any} solution.
For $\MBro$ the result is completely new.

%
%%%%%%%%%%%%%%%%%%%%%%%%%%%%%%%%%%%%%%%%%%%%%%%%%%%%%%%%%%%%%%%%%%%%%%%%%%%%%%%%%%%%%%%%%%%%%%%%%%%%
%%%%%%%%%%%%%%%%%%%%%%%%%%%%%%%%%%%%%%%%%%%%%%%%%%%%%%%%%%%%%%%%%%%%%%%%%%%%%%%%%%%%%%%%%%%%%%%%%%%%
\subsection{Notations and preliminary results}
\label{sec:notations}
%%%%%%%%%%%%%%%%%%%%%%%%%%%%%%%%%%%%%%%%%%%%%%%%%%%%%%%%%%%%%%%%%%%%%%%%%%%%%%%%%%%%%%%%%%%%%%%%%%%%
\paragraph{Canonical processes.}
In the Brenier model and in the Br\"odinger problem we are dealing with generalized flows, which are probability measures on $C([0,1]; \T^d)$.
We will denote by $X = (X_t)_{t \in [0,1]}$ the canonical process on this space. 
Put differently, for all $t \in [0,1]$, the random variable $X_t$ is the evaluation map at time $t$:
\begin{equation*}
X_t: \omega \in C([0,1]; \T^d) \quad \mapsto \quad \omega_t \in \T^d.
\end{equation*} 
Likewise, in the multiphase Brenier and Br\"odinger models, we are dealing with traffic plans, \emph{i-e} with probability measures on $C([0,1]; \P(\T^d))$.
We will denote by $\XX = (\XX_t)_{t \in [0,1]}$ the canonical process on this space. For all $t \in [0,1]$, $\XX_t$ is the evaluation map at time $t$:
\begin{equation*}
\XX_t: \rho \in C([0,1]; \P(\T^d)) \quad \mapsto \quad \rho_t \in \P(\T^d).
\end{equation*} 
\paragraph{Push-forward and disintegration.}
If $\mathcal{X}$ and $\mathcal{Y}$ are two Polish spaces, $p$ is a Borel measure on $\mathcal{X}$, and $\Phi: \mathcal{X} \to \mathcal{Y}$ is measurable, we will denote by $\Phi \pf p$ the push-forward of $p$ by $\Phi$, \emph{i-e} the law of $\Phi$ under $p$.

When there is no ambiguity on the map $\Phi$ to be used, we simply denote by $p^y$ the conditional law $p(\bullet\,|\, \Phi=y) \in \P(\mathcal{X})$.
By virtue of the disintegration theorem, $p^y$ is well defined for $\Phi \pf p$-almost every $y \in \mathcal{Y}$, and concentrates on the fiber $\Phi^{-1}(y)$.
We recall that, by definition, if $\varphi$ is a test function on $\mathcal{X}$ then
\begin{equation}
\label{eq:disintegration}
\int_{\mathcal X} \varphi(x) \D p(x) = \int_{\mathcal Y} \left\{ \int_{\Phi^{-1}(y)} \varphi(x) \D p^y(x)\right\} \D\, (\Phi\pf p) (y),
\end{equation} 
or equivalently:
\begin{equation}
\label{eq:disintegration_measure}
p = \int_{\mathcal{Y}} p^y \D\, (\Phi\pf p) (y).
\end{equation}
If $P$ is the law of a process on $\T^d$ or $\R^d$ (that is, an element of $\P(C[0,1]; \T^d)$ or $\P(C[0,1]; \R^d)$), and if the map $\Phi = X_t$ is the evaluation at time $t$ defined above, we will write $P_t \coloneqq  X_t {}\pf P$ for the marginal of $P$ at time $t$.
Following the standard notations, $P_{0,1}$ will stand for the joint law $(X_0, X_1)\pf P$, and $P^{x,y}$ will refer to the conditional law $P(\bullet\, |\, X_0 = x, X_1=y)$.
These laws will frequently have their diffusivity as a superscript (typically $P^{\nu}$).
In that case, we write
\begin{equation*}
P^{\nu, x, y} \coloneqq  P^{\nu}(\bullet\,|\, X_0 = x, X_1 = y).
\end{equation*}
With these notations, the marginal constraint \eqref{eq:gamma_endpoint} can be reformulated as $P_{0,1} = \gamma$ and the incompressibility \eqref{eq:incompressible_flow} reads $P_t = \Leb$.

Similarly, if $\PP$ is a traffic plan and $t \in [0,1]$, we will write $\PP_t \coloneqq  \XX_t {}\pf \PP$ for the time-$t$ marginal, $\PP_{0,1}$ will stand for the joint law $(\XX_0, \XX_1)\pf \PP$, and $\PP^{\rho_0,\rho_1}$ will refer to the conditional law $P(\bullet \,|\, \XX_0 = \rho_0, \XX_1=\rho_1)$.
These laws will frequently have their diffusivity as a superscript (typically $\PP^{\nu}$).
In that case, we write
\begin{equation*}
\PP^{\nu, \rho_0, \rho_1} \coloneqq  \PP^{\nu}(\bullet\,|\, \XX_0 = \rho_0, \XX_1 = \rho_1).
\end{equation*}
With these notations, \eqref{eq:Mjoint_law} rewrites $\PP_{0,1} = \Gamma$ and the generalized incompressibility \eqref{eq:incomressibility_MREu} reads $\int \rho \D \PP_t(\rho) = \Leb$.
%
%%%%%%%%%%%%%%%%%%%%%%%%%%%%%%%%%%%%%%%%%%%%%%%%%%%%%%%%%%%%%%%%%%%%%%%%%%%%%%%%%%%%%%%%%%%%%%%%%%%%
%
\paragraph{The relative entropy.}
If $\mathcal{X}$ is a Polish space, $r$ is a reference positive and finite Radon measure on $\mathcal{X}$, and $p$ is a Borel probability measure on $\mathcal{X}$, the relative entropy of $p$ with respect to $r$ (also known as the Kullback-Leibler divergence) is defined by
\begin{equation*}
H(p\,|\,r) \coloneqq  \left\{ \begin{aligned}
&\int \log \rho \D p = \int \rho \log \rho \D r &&\mbox{if } p \ll r \mbox{ and } \rho=\frac{\D p}{\D r},\\
&+\infty &&\mbox{else}. 
\end{aligned} \right.
\end{equation*}
Jensen's inequality applied to the convex function $\eta \mapsto  \eta \log \eta$ implies the lower bound $H(p\,|\,r) \geq - \log r(\mathcal{X})$ (which is $0$ when $r$ is a probability measure).
In Section~\ref{section:GammaCv_process}, we will need several elementary results about the relative entropy, listed here without proofs.
The first one concerns the change of reference measure.
\begin{Prop}
	\label{prop:entropy_wrt_fR}
	Let $r$ and $p$ be as above and let $f \in L^1(\mathcal{X}, r)$ be nonnegative and $p$-almost surely positive.
	Then
	\begin{equation*}
	H(p\,|\,f\cdot r)  = H(p\,|\,r) - \int \log f \D p.
	\end{equation*}
\end{Prop}
The behaviour of the relative entropy with respect to disintegration is given by
\begin{Prop}
	\label{prop:disintegration_entropy}
	Let $\mathcal{X}$ and $\mathcal{Y}$ be Polish spaces, $r$ and $p$ be as above and take $\Phi: \mathcal{X} \to \mathcal{Y}$ a measurable map.
	Then with the same notations as before,
	\begin{equation*}
	H(p\,|\,r) = H(\Phi \pf p \,|\, \Phi \pf r) + \int H(p^x \,|\, r^x) \D \Phi \pf p(x).
	\end{equation*} 
\end{Prop}
Finally, if $\Phi$ is one-to-one, then simultaneously pushing forward $r$ and $p$ by $\Phi$ does not change their relative entropy:
\begin{Prop}
	\label{prop:entropy_one_to_one} 
	Take $\mathcal{X}$, $\mathcal{Y}$, $r$, $p$ and $\Phi$ as in Proposition \ref{prop:disintegration_entropy}.
	Assume furthermore that $p \ll r$ and that there exists $\Psi: \mathcal{Y} \to \mathcal{X}$ such that $r$-almost surely, $\Psi \circ \Phi = \Id_{\mathcal{X}}$.
	Then
	\begin{equation*}
	H(p\,|\,r) = H(\Phi \pf p \,|\, \Phi \pf r).
	\end{equation*}
\end{Prop}

For probability measures on the torus $\rho\in \P(\T^d)$ and if no confusion arises, we simply write
$$
H(\rho)\coloneqq  H(\rho\, | \, \Leb)=\int_{\T^d} \rho(x)\log \rho(x) \D x
$$
for the entropy computed relatively to the Lebesgue measure.
(Once again, we keep the same notation $\rho$ for a measure and its density with respect to $\Leb$.)
\paragraph{The heat flow.}  Let us denote by $\tau_s$ the heat kernel in the torus
\begin{equation*}
\partial_s\tau_s=\frac 12\Delta\tau_s
\end{equation*}
at time $s>0$, started from the initial Dirac distribution $\tau_0=\delta$.
We will use the following Gaussian estimate several times:
\begin{Lem}
	There are two dimensional constants $k_d,K_d>0$ such that for all $s \in (0,1]$ and all $x,y \in \T^d$,
	\begin{equation}
	\label{eq:estim_heat_flow} 
	\frac{k_d}{\sqrt{2 \pi s}^d}  \exp\left( -\frac{\dist^2(x, y)}{2 s} \right)
	\leq \tau_s(y-x)
	\leq \frac{K_d}{\sqrt{2 \pi s}^d} \exp \left( -\frac{\dist^2(x, y)}{2 s}\right).
	\end{equation} 
\end{Lem}

This type of results can be obtained under general assumptions on the domain and we refer \emph{e.g.} to \cite{varadhan1967diffusion,li1986parabolic,grigor1999estimates} for this delicate topic.
In the torus we have the explicit formula
\begin{equation}
\label{eq:explicit_heat_kernel}
\forall x,y \in \T^d, \qquad \tau_s(y-x) = \frac{1}{\sqrt{2 \pi s}^d} \sum_{\bar l \in \Z^d} \exp\left( - \frac{|{\bar y}-{\bar x} + \bar l|^2}{2 s} \right) ,
\end{equation}
where $\bar x, \overline{y} \in \R^d$ are chosen so that $\pi( \overline{x} ) = x$ and $\pi(\overline{y}) = y$.
Hence in this particular case the bounds \eqref{eq:estim_heat_flow} could be worked out by hand.
As such, the upper bound can only be valid for short times (note that we took care to assume $s \leq 1$ in our statement) and indeed we shall only use this in the limit $s \to 0$.
%
%
%%%%%%%%%%%%%%%%%%%%%%%%%%%%%%%%%%%%%%%%%%%%%%%%%%%%%%%%%%%%%%%%%%%%%%

%%%%%%%%%%%%%%%%%%%%%%%%%%%%%%%%%%%%%%%%%%%%%%%%%%%%%%%%%%%%%%%%%%%%%%%%%%%%%%%%%%%%%%%%%%%%%%%%%%%%%%%%%%%%%%%%%%%%%%%%
%%%%%%%%%%%%%%%%%%%%%%%%%%%%%%%%%%%%%%%%%%%%%%%%%%%%%%%%%%%%%%%%%%%%%%%%%%%%%%%%%%%%%%%%%%%%%%%%%%%%%%%%%%%%%%%%%%%%%%%%

\section{Convergence of $\Bro$ towards $\REu$}
\label{section:GammaCv_process}
The goal of this section is to prove Theorem~\ref{thm:GammaCv_process}.
In fact the statement will be obtained as a consequence of the following stronger, but more technical result:
\begin{Thm}
\label{thm:general_GammaCv_process}
~
\begin{enumerate}
	\item \label{item:general_Gamma-liminf} Let $(P^{\nu})_{\nu >0}$ be a sequence of incompressible generalized flows narrowly converging to $P$.
	Then
	\begin{equation*}
	\liminf_{\nu \to 0}\, \Hbar_{\nu} (P^{\nu}) \geq \Abar(P).
	\end{equation*}
	\item 
	\label{item:general_Gamma-limsup}
	Let $P$ be an admissible generalized flow for $\REu(\gamma)$, and $\gamma^\nu\narrowcv\gamma$.
	The following are equivalent:
	\begin{enumerate}
	\item
	\label{item:general_recovery_flow_P_nu}
	there exists a sequence of generalized incompressible flows $P^\nu$  with marginals $P^\nu_{0,1}=\gamma^\nu$ that narrowly converges to $P$ and such that
	\begin{equation}
	\label{eq:Gamma_limsup_Bro}
	\limsup_{\nu \to 0}\, \Hbar_{\nu}(P^{\nu}) \leq \Abar(P)
	\end{equation}
	\item
	\label{item:general_recovery_marginal_gamma_nu}
	the sequence of marginals satisfies
	\begin{equation}
	\label{eq:recovery_Sch_OT}
	\limsup_{\nu \to 0 }\, \nu H(\gamma^\nu \,|\, R^{\nu}_{0,1}) \leq \frac{1}{2}\int \dist(x,y)^2 \D \gamma(x,y).
	\end{equation}
	\end{enumerate}
\end{enumerate}	
\end{Thm}
\begin{Rem}
	Condition \eqref{eq:recovery_Sch_OT} exactly requires $(\gamma^\nu)_{\nu>0}$ to be a recovery sequence for the $\Gamma$-convergence
	\begin{equation}
	\label{eq:Gamma_CV_entropicOT_OT}
	\Gamma-\lim\limits_{\nu\to 0}\,\big\{\nu H(\bullet \, |\, R^{\nu}_{0,1}) + \iota_\mathrm{Bis} \big\}
	= C_{\mathsf {MK}} + \iota_\mathrm{Bis},
	\end{equation}
	where:
	\begin{equation*}
	C_{\mathsf {MK}}(\gamma) \coloneqq \frac{1}{2} \int \dist(x,y)^2 \D \gamma(x,y)
	\end{equation*}
	is the Monge-Kantorovich quadratic cost functional and $\iota_\mathrm{Bis}(\gamma)$ is the characteristic function of the bistochasticity constraint.
	The $\Gamma$-convergence \eqref{eq:Gamma_CV_entropicOT_OT} is well known as a particular case of results from \cite{leonard2012schrodinger} and \cite{carlier2017convergence}\footnote{
	\label{footnote:dvpt_log}
	In \cite{carlier2017convergence}, our term $\nu H(\gamma\,|\, R^{\nu}_{0,1})$ is replaced by $C_{\mathsf{MK}}(\gamma) + \nu H(\gamma \,|\, \Leb \otimes \Leb)$ and the analysis is carried out in the whole space $\R^d$.
	In that setting one has the explicit formula $R^\nu_{0,1} = \exp\left( -\frac{|x-y|^2}{2\nu} \right) / \sqrt{2\pi\nu}^d \D x \otimes \D y$, and expanding the logarithms in Proposition \ref{prop:entropy_wrt_fR} gives exactly
	$\nu H(\gamma\,|\, R^{\nu}_{0,1})=\nu\log((2\pi\nu)^{d/2})+C_{\mathsf{MK}}(\gamma) + \nu H(\gamma \,|\, \Leb \otimes \Leb)$.
	Consequently, the corresponding optimization programs are the same.
	We also note that the construction given in \cite{carlier2017convergence} is easily adapted to the torus.
	},
	hence for a given $\gamma$ there always exists a sequence $\gamma^\nu$ as in \eqref{eq:recovery_Sch_OT} and in practice our Theorem~\ref{thm:general_GammaCv_process} guarantees that one can always construct a recovery sequence $P^\nu\narrowcv P$ in \eqref{eq:Gamma_limsup_Bro}.
\end{Rem}
Before going into the technical details we will need a few preliminary definitions and results.
In order not to interrupt the flow of the exposition, we postpone the proofs of the latter to the appendix.

Observe first that the kinetic action $A(\omega)=\frac 12 \int_0^1|\dot\omega_t|^2\D t$ is only lower semicontinuous for the uniform topology on $C([0,1];\T^d)$, and unbounded.
In order to circumvent this lack of regularity we will approximate $A$ by difference quotients as follows: for large $N\in\N$ we set $\tau=1/N$, $t_n=n\tau$ for $n=0,\dots,N$, and for any $\omega\in C([0,1];\T^d)$ we define
\begin{equation}
\label{eq:defAN}
A_N(\omega)\coloneqq \frac{1}{2\tau}\sum\limits_{n=0}^{N-1}\dist^2(\omega_{t_n},\omega_{t_{n+1}}).
\end{equation}
This is a a good approximation of $A$ in the sense that, for any fixed $\omega$ and by classical properties of difference quotients in the one-dimensional Sobolev space $AC^2([0,1];\T^d)=H^1([0,1];\T^d)$, we have pointwise convergence
$$
A_N(\omega)\to A(\omega) \quad \mbox{as } N\to\infty
$$
(in particular $A_N(\omega)\to +\infty$ if $\omega\not\in AC^2$).
Note moreover that, for fixed $N$, $A_N$ is continuous for the uniform topology {\color{red}(as the sum of finitely many evaluations)} and bounded, since in the torus $\dist^2(\omega_{t_n},\omega_{t_{n+1}})\leq diam(\T^d)^2< +\infty$. (In the whole space $\R^d$ one should replace $A_N$ by its truncation $\tilde A_N\coloneqq \min(A_N, N)$ to guarantee boundedness, and the rest of the argument below then applies \emph{mutatis mutandis}.)
Our first technical statement reads
\begin{Lem}
	\label{lem:estimation_int_AN_Bxy}
	Consider $A_N$ as defined in \eqref{eq:defAN}, and let $R^\nu$ denote the reversible Brownian motion with diffusivity $\nu$.
	For all $\alpha\in (0,1)$ and $\nu\in (0,1)$ there holds
	\begin{equation}
	\label{eq:estimation_int_AN_B}
	\int \exp\left(\frac{\alpha}{\nu}A_N(\omega)\right) \D R^{\nu}(\omega)\leq \frac{1}{(1-\alpha)^{Nd/2}}.
	\end{equation}
	Similarly, for all $x,y\in\T^d$ let $R^{\nu,x,y} \coloneqq  R^{\nu}(\bullet \, | \, X_0 = x, \, X_1 = y)$ be the Brownian bridge joining $x$ to $y$.
	There is a dimensional constant $C_d>0$ such that
	\begin{equation}
	\label{eq:estimation_int_AN_Bxy}
	\int \exp\left(\frac{\alpha}{\nu}A_N(\omega)\right) \D R^{\nu,x,y}(\omega)\leq \frac{C_d}{(1-\alpha)^{Nd/2}}\exp\left(\frac{\alpha}{2\nu}\dist^2(x,y)\right).
	\end{equation}
\end{Lem}
\noindent
In the whole space this would readily follow from explicit computations for Gaussian vectors. 
Working in the torus makes the proof slightly more intricate, and we postpone the proof to Appendix~\ref{section:brownian_torus} for convenience.\\

We also need to define a notion of translated Brownian bridges, which will play a crucial role in our construction of the recovery sequence $(P^{\nu})$. 
To do so, let us first denote by $\Pi$ the projection
\begin{equation*}
\Pi: \quad  \omega \in C([0,1]; \R^d) \quad  \mapsto \quad  \Big( t \mapsto \pi(\omega_t) \Big) \in C([0,1]; \T^d). 
\end{equation*}
Then, if $\omega \in C([0,1]; \T^d)$, we write $T_{\omega}$ for the translation map
\begin{equation}
\label{eq:def_translation}
T_{\omega}:\quad  \alpha \in C([0,1]; \T^d) \quad \mapsto \quad  \omega + \alpha \in C([0,1]; \T^d).
\end{equation}
Let $\Bbar^{\nu}\coloneqq \overline R^{\nu,0,0}$ be the standard Brownian bridge in $\R^d$ with diffusivity $\nu$ and joining $0$ to $0$.
Our translated bridges are defined as follows:
\begin{Def}
	\label{def:B^nu_omega}
	If $\omega \in C([0,1]; \T^d)$ and $\nu>0$, we set
	\begin{equation*}
	B^{\nu}_{\omega} \coloneqq  T_{\omega} {}\pf \Pi\pf \Bbar^{\nu}.
	\end{equation*}
\end{Def}
Roughly speaking, $B^{\nu}_{\omega}$ is obtained by adding the projection of the Brownian bridge to $\omega$.
Note however that the Brownian bridge in the torus is \emph{not} the projection of the Brownian bridge in $\R^d$, \emph{i.e} $R^{\nu, 0, 0} \neq \Pi \pf \Bbar^{\nu}$.
As a consequence, $B^{\nu}_{\omega} \neq T_{\omega} {}\pf R^{\nu, 0,0}$.
This alternative choice of translated bridges would have made the proof of Lemma~\ref{lem:entropy_translation_bridges_torus} below more delicate.

The entropy of $B^{\nu}_{\omega}$ with respect to the bridges of $R^{\nu}$ will be computed thanks to
\begin{Lem}
	\label{lem:entropy_translation_bridges_torus}
	There exists a dimensional constant $C = C_d>0$ such that, for all $\nu \leq 1$ and all $\omega \in C([0,1]; \T^d)$,
	\begin{equation}
	\label{eq:entropy_translation_bridges_torus}
	\nu H(B^{\nu}_{\omega} \, | \, R^{\nu, \omega_0, \omega_1}) \leq A(\omega) - \frac{1}{2} \dist^2(\omega_0,\omega_1) + C \nu.
	\end{equation}
\end{Lem}
\noindent
Once again, we postpone the proof to the appendix.
Note that, in the whole space $\R^d$, the corresponding result is stated in Lemma \ref{lem:entropy_translation_bridge} and follows from the classical Cameron-Martin formula.
In that case, equality holds in \eqref{eq:entropy_translation_bridges_torus} with $C = 0$.
\\

With this preliminary material at hand, let us first prove Theorem \ref{thm:GammaCv_process} using Theorem \ref{thm:general_GammaCv_process}:
\begin{proof}[Proof of Theorem \ref{thm:GammaCv_process}]
	Both $\Gamma-\liminf$ parts are a direct consequence of point \ref{item:general_Gamma-liminf} of Theorem \ref{thm:general_GammaCv_process}, regardless of any marginal constraint.
	
	For the $\Gamma-\limsup$ part in point \ref{item:GammaCv_no_marge}, fix an admissible generalized flow $P$ with marginals $P_{0,1}=\gamma$.
	By \cite[Thm.~2.7]{carlier2017convergence} there always exists a recovery sequence $\gamma^\nu\narrowcv \gamma$ for the optimal transport problem, \emph{i-e} satisfying \eqref{eq:recovery_Sch_OT} (see Footnote \ref{footnote:dvpt_log}).
	Thus by Theorem~\ref{thm:general_GammaCv_process} we can construct a recovery sequence $P^\nu\narrowcv P$ satisfying \eqref{eq:Gamma_limsup_Bro}.
	
	For the $\Gamma-\limsup$ part of point \ref{item:GammaCv_with_marge}, we claim that the particular sequence $\gamma^\nu=\gamma$ satisfies \eqref{eq:recovery_Sch_OT}.
	On this premise, Theorem~\ref{thm:general_GammaCv_process} immediately provides a recovery sequence $P^\nu\narrowcv P$ with marginals $P^\nu_{0,1}=\gamma$ and satisfying \eqref{eq:Gamma_limsup_Bro} as required, hence it suffices to check our claim.
	To this end, observe that the density $r^{\nu}_{0,1}$ of $R^{\nu}_{0,1}$ with respect to $\Leb \otimes \Leb$ is
	\begin{equation*}
	r^{\nu}_{0,1}(x,y) = \tau_{\nu}(y-x)
	\end{equation*}
	where as before, $(\tau_s)$ is the heat kernel in the torus at time $s$. 
	By Proposition~\ref{prop:entropy_wrt_fR} (with $f=r^\nu_{0,1}$ and $r=\Leb\otimes\Leb$) we have thus
	\begin{align*}
	 \nu H(\gamma\,|\, R^{\nu}_{0,1})&=\nu H(\gamma\,|\,\Leb\otimes\Leb)
	 - \int \nu \log r^{\nu}_{0,1}(x,y)\D\gamma(x,y)\\
	&=\nu H(\gamma\,|\,\Leb\otimes\Leb)
	 - \int \nu \log \tau_{\nu}(y-x) \D\gamma(x,y)\\
	 & \xrightarrow[\nu\to 0]{}
	 0 + \frac 12 \int \dist^2(x,y)\D\gamma(x,y),
	\end{align*}
	where the last line is obtained using our assumption $H(\gamma \, | \, \Leb \otimes \Leb) < + \infty$ as well as \eqref{eq:estim_heat_flow} (which implies in particular $\nu\log \tau_\nu(y-x)\to-\dist^2(x,y)$ uniformly on $\T^d\times\T^d$).
	Hence our claim holds and the proof is complete.
	\qed
\end{proof}

Let us now carry on with the proof of Theorem \ref{thm:general_GammaCv_process}, which will go through several steps.
We begin with
\begin{proof}[Proof of point \ref{item:general_Gamma-liminf} of Theorem \ref{thm:general_GammaCv_process}]
This result is weaker than the same statement without the incompressibility constraint.
Hence, it is direct a consequence of the works of C. L\'eonard, see \emph{e.g.} \cite[Prop.~2.5]{leonard2012schrodinger}.
For the sake of self-completeness we choose to present here an independent proof, fully leveraging the explicit structure of the reversible Brownian motion $R^\nu$ as a particular reference measure (whereas C. L\'eonard covers much more general settings).
We will also recycle part of the argument later on in the proof of \ref{item:general_recovery_flow_P_nu}$\implies$\ref{item:general_recovery_marginal_gamma_nu} in Theorem~\ref{thm:general_GammaCv_process}, hence we give the full details.

Consider a sequence $P^\nu\narrowcv P$ as in our statement.
Observe first that the Legendre transform of $h(u)=u\log u$ is $h^*(v)=e^{v-1}$, in particular
\begin{equation}
\label{eq:ineq_convexity_Legendre_log}
u\log u\geq uv-e^{v-1}\qquad \mbox{for all }u,v.
\end{equation}
Fix any bounded and continuous function $f(\omega)$ on $C([0,1];\T^d)$.
Taking $u=\frac{dP^\nu}{dR^\nu}(\omega),v=f(\omega)/\nu$ in the previous convexity inequality and integrating with respect to $R^\nu$, we get
\begin{align}
 \nu\Hbar_\nu(P^\nu)
 \notag &=\nu \int \frac{\D P^\nu}{\D R^\nu}(\omega)\log\frac{\D P^\nu}{\D R^\nu}(\omega) \D R^\nu (\omega)\\
 \notag &\geq \nu \int \frac{\D P^\nu}{\D R^\nu}(\omega) \frac{f(\omega)}{\nu}\D R^\nu(\omega)  -  \nu \int \exp\left(\frac{f(\omega)}{\nu}-1\right)\D R^\nu(\omega)\\
 \label{eq:ineq_Legrendre_Pnu} &= \int f(\omega)\D P^\nu(\omega) -\nu e^{-1}\int \exp\left(\frac{f(\omega)}{\nu}\right)\D R^\nu(\omega).
\end{align}
Ideally, one wishes to test $f=A$ (the kinetic action \eqref{eq:def_kinetic_action_path}) in this formula, and pass to the limit $\nu\to 0$ hoping that the exponential term $\nu e^{-1}\int\{\dots\}\to 0$ to conclude that $\liminf \nu\Hbar_\nu(P^\nu)\geq \int A \D P$.
However this is not possible since Brownian paths are nowhere differentiable hence $A(\omega)=+\infty$ for $R^\nu$ and $P^{\nu}$-almost all paths $\omega$. 
Instead, we use as a surrogate the difference quotient approximation $A_N$ defined in \eqref{eq:defAN}.

For technical reasons let us fix a parameter $\alpha\in(0,1)$ close to $1$, and recall that for large $N$ we write $\tau=1/N$ and $t_n=n\tau$.
Taking $f=\alpha A_N$ in \eqref{eq:ineq_Legrendre_Pnu}, the last integral in the r.h.s. is immediately bounded by Lemma~\ref{lem:estimation_int_AN_Bxy} as
$$
 \int \exp\left(\alpha \frac{A_N(\omega)}{\nu}\right)\D R^\nu(\omega)
 \leq \frac{1}{(1-\alpha)^{Nd/2}}.
$$
As a consequence we get
\begin{align*}
 \liminf\limits_{\nu\to 0}\,\nu\Hbar_\nu(P^\nu)
 & \geq \liminf\limits_{\nu\to 0}
 \left\{
 \alpha\int A_N(\omega)\D P^\nu(\omega)
 \right\}\\
 &\qquad -
 \limsup\limits_{\nu\to 0}
 \left\{
 \nu e^{-1}\int \exp\left(\alpha \frac{A_N(\omega)}{\nu}\right)\D R^\nu(\omega)
 \right\}\\
 & \geq \liminf\limits_{\nu\to 0}
 \left\{
 \alpha\int A_N(\omega)\D P^\nu(\omega)
 \right\} -
 \limsup\limits_{\nu\to 0}
 \left\{
 \frac{\nu e^{-1}}{(1 - \alpha)^{Nd/2}}\right\}\\
 & = \alpha\int A_N(\omega)\D P(\omega)-0
\end{align*}
because we assumed $P^\nu\narrowcv P$ (and $A_N$ is bounded continuous).
Taking next $\alpha\to 1$ and then $\liminf$ as $N\to\infty$ with pointwise convergence $A_N(\omega)\to A(\omega)$ for \emph{all} $\omega\in C([0,1];\T^d)$, we finally get by Fatou's lemma
$$
\liminf\limits_{\nu\to 0}\,\nu\Hbar_\nu(P^\nu)
\geq 1\cdot\liminf\limits_{N\to\infty}\int A_N(\omega )\D P(\omega)
\geq \int A(\omega)\D P(\omega)
=\A( P)
$$
and the proof is complete.
\qed
\end{proof}

We proceed now with the proof of the equivalence \ref{item:general_recovery_flow_P_nu}$\iff$\ref{item:general_recovery_marginal_gamma_nu} in Theorem~\ref{thm:general_GammaCv_process}.
In order to ease the exposition we opted for dividing the argument in two steps, one for each implication.
\begin{proof}[Proof of Theorem~\ref{thm:general_GammaCv_process},  \ref{item:general_recovery_flow_P_nu}$\implies$\ref{item:general_recovery_marginal_gamma_nu}]
 Assume that $P^\nu$ converges to $P$ and satisfies the $\Gamma-\limsup$ inequality \eqref{eq:Gamma_limsup_Bro}.
Disintegrating with respect to $(X_0,X_1)$, we have by Proposition~\ref{prop:disintegration_entropy}
 $$
 \Hbar_\nu(P^\nu)
 =  \nu H(\gamma^\nu\, | \, R^\nu_{0,1}) +  \nu\int H(P^{\nu,x,y}\, | \, R^{\nu,x,y})\D \gamma^\nu(x,y).
 $$
 Hence with our assumption \eqref{eq:Gamma_limsup_Bro} there holds
 \begin{align}
\notag \limsup\limits_{\nu\to 0}&\,\{\nu H(\gamma^\nu \, | \, R^\nu_{0,1}) \} \\
 \notag &\leq \limsup \limits_{\nu\to 0}\, \Hbar_\nu(P^\nu)
 -\liminf\limits_{\nu\to 0}\, \left\{
 \nu\int H(P^{\nu,x,y} \, | \, R^{\nu,x,y})\D \gamma^\nu(x,y)
 \right\}
 \\
  \label{eq:marginals_gamma_nu_entropy_disintegrated} &\leq \Abar(P)  -  \liminf\limits_{\nu\to 0}\, \left\{  \nu\int H(P^{\nu,x,y} \, | \, R^{\nu,x,y})\D \gamma^\nu(x,y)   \right\}.
 \end{align}
 In order to estimate the last integral we proceed using the same strategy as in the proof of point \ref{item:general_Gamma-liminf} of Theorem~\ref{thm:general_GammaCv_process} earlier.
 For large $N\in\mathbb N$ and $\tau=1/N$ we write again $t_n=n\tau$ for $n=0\dots N$, and consider now
 \begin{equation*}
 \hat A_N(\omega)\coloneqq 
 \frac{1}{2\tau}\sum\limits_{n=0}^{N-1}\dist^2(\omega_{t_n},\omega_{t_{n+1}}) - \frac{1}{2}\dist^2(\omega_{0},\omega_1)
 = A_N(\omega)- \frac{1}{2}\dist^2(\omega_{0},\omega_1).
\end{equation*} 
This function is continuous for the uniform topology on $C([0,1];\T^d)$, bounded, and it is as before a good approximation of
$$
\hat A(\omega)\coloneqq 
\left\{
\begin{array}{ll}
\frac{1}{2}\int_0^1|\dot\omega_t|^2\D t -\frac 12 \dist^2(\omega_0,\omega_1) & \mbox{if }\omega\in AC^2\\
+\infty & \mbox{otherwise}
\end{array}
\right. 
$$
as $N\to\infty$.
We fix as before a parameter $\alpha\in(0,1)$ close to $1$.
 Exploiting the convexity inequality \eqref{eq:ineq_convexity_Legendre_log} with $u=\frac{d P^{\nu,x,y}}{d R^{\nu,x,y}}(\omega)$ and $v=\frac{\alpha}{\nu}\hat A_N(\omega)$, integrating first with respect to $R^{\nu,x,y}$ (note that $\omega_0=x$ and $\omega_1=y$ for $R^{\nu,x,y}$-almost all $\omega$) and then with respect to $\gamma^\nu$, we get
 \begin{multline}
  \nu\int H(P^{\nu,x,y} \, | \, R^{\nu,x,y})\D \gamma^\nu(x,y) 
  \geq  \alpha  \iint \hat A_N(\omega)\D P^{\nu,x,y}(\omega)\D\gamma^\nu(x,y)\\
  - \nu e^{-1} \int \exp\left(-\frac{\alpha}{2\nu} \dist(x,y)^2 \right)\int \exp\left(\frac{\alpha}{\nu}A_N(\omega)\right) \D R^{\nu,x,y}(\omega)\D\gamma^\nu(x,y).
   \label{eq:ineq_Legrendre_Pnuxy}
 \end{multline}
 Since by assumption $P^\nu\narrowcv P$, and because $\hat A_N$ is bounded continuous, we see that the first term in the r.h.s
 \begin{equation}
 \label{eq:cv_linear_part}
 \iint \hat A_N(\omega)\D P^{\nu,x,y}(\omega)\D\gamma^\nu(x,y) = \int \hat A_{N}(\omega) \D P^{\nu}(\omega) \underset{\nu \to 0}{\longrightarrow} \int \hat A_{N}(\omega) \D P(\omega),
 \end{equation}
 thus it only remains to show that the $\limsup$ of the exponential term in \eqref{eq:ineq_Legrendre_Pnuxy} goes to zero (just like in the previous proof of point \ref{item:general_Gamma-liminf} of Theorem \ref{thm:general_GammaCv_process}).
This is where we need to use Lemma~\ref{lem:estimation_int_AN_Bxy}. From \eqref{eq:ineq_Legrendre_Pnuxy}\eqref{eq:cv_linear_part}\eqref{eq:estimation_int_AN_Bxy}, we get:
\begin{align*}
&\liminf_{\nu \to 0}  \left\{\nu\int H(P^{\nu,x,y} \, | \, R^{\nu,x,y})\D \gamma^\nu(x,y) \right\}\\
&\hspace{0.5cm}\geq \alpha \int \hat A_N(\omega) \D P(\omega) \\
& \hspace{1cm}
- \limsup_{\nu \to 0} \left\{
\nu e^{-1} \int \exp\left(-\frac{\alpha}{2\nu} \dist^2(x,y) \right)\frac{C_d \exp\left(\frac{\alpha}{2\nu}\dist^2(x,y)\right)}{(1-\alpha)^{Nd/2}}\D\gamma^\nu(x,y)
\right\}\\
&\hspace{0.5cm}\geq \alpha \int \hat A_N(\omega) \D P(\omega) - \limsup_{\nu \to 0} \left\{\nu e^{-1} \times \frac{C_d}{(1 - \alpha)^{Nd/2}}\right\}\\
&\hspace{0.5cm}= \alpha \int \hat A_N(\omega) \D P(\omega) .
\end{align*}
 Taking next $\alpha\to 1$ and $\liminf$ as $N\to\infty$ with now
 $$
 \hat A_N(\omega)=A_N(\omega)-\frac{1}{2}\dist^2(\omega_0,\omega_1)\to A(\omega)-\frac{1}{2}\dist^2(\omega_0,\omega_1)
 $$
 for all $\omega$, we conclude by Fatou's lemma that
 \begin{multline*}
 \liminf\limits_{\nu\to 0}\,\left\{\nu\int H(P^{\nu,x,y} \, | \, R^{\nu,x,y})\D \gamma^\nu(x,y)\right\}\\
 \geq \int \left\{A(\omega)-\frac{1}{2}\dist^2(\omega_0,\omega_1)\right\}\D P(\omega)=\Abar(P)-\frac{1}{2}\int \dist^2(x,y)\D\gamma(x,y).
 \end{multline*}
Substituting into \eqref{eq:marginals_gamma_nu_entropy_disintegrated} finally gives \eqref{eq:recovery_Sch_OT} and concludes the proof.
\qed
\end{proof}
Let us now establish the converse implication:
\begin{proof}[Proof of Theorem~\ref{thm:general_GammaCv_process}, \ref{item:general_recovery_marginal_gamma_nu}$\implies$\ref{item:general_recovery_flow_P_nu}]
	Take $(\gamma^\nu)_{\nu>0}$ as in our statement.
	Since we have $\gamma^\nu \narrowcv \gamma$, a closer look into the proof of \cite[Thm.~1]{baradat2018continuous} (continuity of the optimal action in Brenier's problem $\REu(\gamma)$ with respect to the marginal $\gamma$) gives a sequence $Q^{\nu}$ of generalized flows converging to $P$ such that $Q^{\nu}$ is admissible for $\REu(\gamma^\nu)$ and 
	\begin{equation}
	\label{eq:cv_action}
	\lim_{\nu \to 0} \Abar(Q^{\nu}) = \Abar(P).
	\end{equation}
	Let now
	\begin{equation}
	\label{eq:def_P^nu}
	P^{\nu} \coloneqq  \int B^{\nu}_{\omega} \D Q^{\nu}(\omega),
	\end{equation}
	with $B^{\nu}_{\omega}$ as in Definition~\ref{def:B^nu_omega}.
	Roughly speaking, $P^\nu$ is a noisy version of $Q^\nu$, where all the paths initially charged by $Q^\nu$ receive now an additional small Brownian perturbation.
	
	First of all, we claim that $P^{\nu} \narrowcv P$ as $\nu\to 0$.
	Indeed, if $\varphi$ is a test function on $C([0,1]; \T^d)$, let us check that
	\begin{equation*}
	\varphi^\nu: \omega \, \mapsto \, \int \varphi(\alpha) \D B^{\nu}_{\omega}(\alpha)
	\end{equation*}
	converges uniformly towards $\varphi$ on the compact sets of $C([0,1]; \T^d)$.
	Indeed if $K$ is such a compact set, take  $m:\R_+ \to \R_+$ a modulus of continuity of $\varphi|_K$.
	Of course, $m$ can be chosen continuous and bounded with $m(0)=0$.
	Then, for all $\omega \in K$,
	\begin{align*}
	|\varphi^{\nu}(\omega) - \varphi(\omega) | &\leq \int |\varphi(\alpha) - \varphi(\omega)| \D B^{\nu}_{\omega}(\alpha)\\
	&= \int |\varphi(\omega + \alpha) - \varphi(\omega)| \D B_0^{\nu}(\alpha)
	\leq \int m(\| \alpha \|_{\infty}) \D B^{\nu}_0(\alpha).
	\end{align*} 
	Since $m$ is bounded and continuous so is the map $\alpha\mapsto m(\|\alpha\|_\infty)$ for the uniform topology on $C([0,1]; \T^d)$, and it is therefore an admissible test-function for the narrow convergence $B^{\nu}_0 \narrowcv \delta_0$ in $\mathcal M(C([0,1]; \T^d))$.
	As a consequence the right-hand side converges to $\int m(\| \alpha \|_{\infty}) \D \delta_0(\alpha)=m(0)=0$ and therefore $\|\varphi^\nu-\varphi\|_\infty\to 0$ as required.
	Since $Q^{\nu} \narrowcv P$ we get now
	\begin{align*}
	\int \varphi(\alpha)\D P^\nu(\alpha)
	&=\int\left(\int\varphi(\alpha)\D B^\nu_\omega(\alpha)\right)\D Q^\nu(\omega)\\
	&= \int \varphi^{\nu}(\omega) \D Q^{\nu}(\omega)
	\quad \xrightarrow[\nu \to 0]{} \quad
	\int \varphi(\omega)\D P(\omega)
	\end{align*}
	as claimed.
	
	Moreover, it is straightforward to check that the marginal constraint is satisfied, $(X_0, X_1)\pf P^{\nu} = \gamma^\nu$.
	Let us check now the incompressibility.
	Take $\varphi$ a test function on $\T^d$: by Definition~\ref{def:B^nu_omega} of $B^\nu_\omega$ and by incompressibility of $Q^{\nu}$, we have for all $t\in[0,1]$
	\begin{align*}
	\int \varphi(\alpha_t)\D P^\nu (\alpha)
	&= \iint \varphi(\alpha_t)\D B^\nu_\omega(\alpha) \D Q^{\nu}(\omega) \\
	&= \iint \varphi(\pi(\overline\alpha_t)+\omega_t)\D \overline B{}^\nu(\overline\alpha) \D Q^{\nu}(\omega)\\
	&= \iint \varphi(\pi(\overline\alpha_t)+\omega_t) \D Q^{\nu}(\omega)\D  \overline B{}^\nu(\overline\alpha)\\
	&= \iint \varphi(\pi(\overline\alpha_t)+x) \D x \D \overline B{}^\nu(\overline\alpha)\\
	&= \iint \varphi(x')  \D x'\,\D\overline B{}^\nu(\overline\alpha)
	= \int \varphi(x')   \D x'.
	\end{align*}
	Next, let us estimate $\Hbar_\nu(P^\nu)=\nu H(P^{\nu}\, | \,  R^{\nu})$.
	Conditioning on the endpoints, we get by Proposition~\ref{prop:disintegration_entropy}
	\begin{equation}
	\label{eq:disintegration_endpoints_P^nu}
	H(P^{\nu} \, | \,  R^{\nu}) = H(\gamma^\nu \, | \, R^{\nu}_{0,1}) + \int H(P^{\nu,x,y} \, | \, R^{\nu,x,y} ) \D \gamma^\nu(x,y).
	\end{equation}
	Moreover, conditioning \eqref{eq:def_P^nu}, we get for $\gamma^\nu$-almost all $(x,y)\in \T^d\times\T^d$:
	\begin{equation*}
	P^{\nu, x,y} = \int B^{\nu}_{\omega} \D Q^{\nu, x, y}(\omega).
	\end{equation*}
	In particular for $\gamma^\nu$-almost all $(x,y) \in \T^d\times\T^d$, by Jensen's inequality, and because $Q^{\nu, x, y}$-almost surely $\omega_0 = x$ and $\omega_1 = y$, there holds
	\begin{equation*}
	H(P^{\nu, x,y} \, |\, R^{\nu, x, y}) \! \leq \! \int \hspace{-2pt} H(B^{\nu}_{\omega} \, |\,  R^{\nu, x, y}) \D Q^{\nu, x, y}(\omega) \! =\hspace{-3pt} \int \hspace{-2pt} H(B^{\nu}_{\omega} \, | \, R^{\nu, \omega_0, \omega_1}) \D Q^{\nu, x, y}(\omega).
	\end{equation*}
	Substituting this inequality in formula \eqref{eq:disintegration_endpoints_P^nu}, we get:
	\begin{align*}
	H(P^{\nu} \, | \, R^{\nu}) &\leq 
	H(\gamma^\nu \,| \, R^{\nu}_{0,1}) + \iint H(B^{\nu}_{\omega} \, | \, R^{\nu, \omega_0, \omega_1}) \D Q^{\nu, x, y}(\omega)\D \gamma^\nu(x,y)
	\\
	&=	H(\gamma^\nu \, | \, R^{\nu}_{0,1}) + \int H(B^{\nu}_{\omega}\,|\, R^{\nu,\omega_0, \omega_1} ) \D Q^{\nu}(\omega).
	\end{align*}
Multiplying by $\nu$ and using \eqref{eq:entropy_translation_bridges_torus}, we get for $\nu <1$:
\begin{align*}
	\nu H(P^{\nu} \, | \, R^{\nu}) &\leq \nu H(\gamma^\nu \, | \, R^{\nu}_{0,1}) + \int \left\{ A(\omega) - \frac{1}{2} \dist^2(\omega_0, \omega_1) + C \nu \right\} \D Q^{\nu}(\omega)\\
	&= \nu H(\gamma^\nu \, | \, R^{\nu}_{0,1}) + \Abar(Q^{\nu}) - \frac{1}{2} \int \dist^2(x,y) \D \gamma^{\nu}(x,y) + C \nu,
\end{align*}
where $C$ is a dimensional constant.
	With our assumption \eqref{eq:recovery_Sch_OT} and by \eqref{eq:cv_action}, we finally obtain
	\begin{align*}
	\limsup_{\nu \to 0}& \,\Hbar_{\nu}  (P^{\nu}) \\
	&= \limsup_{\nu \to 0}\, \nu H(P^{\nu} \, | \, R^{\nu}) \\
	&\leq \limsup_{\nu \to 0} \,\left\{ \nu H(\gamma^\nu \, | \, R^{\nu}_{0,1}) + \Abar(Q^{\nu}) - \frac{1}{2} \int \dist^2(x,y) \D \gamma^{\nu}(x,y) + C \nu \right\}\\
	&= \limsup_{\nu \to 0 } \,\nu H(\gamma^\nu \, | \, R^{\nu}_{0,1}) + \lim_{\nu \to 0} \Abar(Q^{\nu}) - \lim_{\nu \to 0}\frac{1}{2} \int \dist^2(x,y) \D \gamma^{\nu}(x,y)\\
	&\leq \frac{1}{2}\int \dist(x,y)^2\D\gamma(x,y)  +  \Abar(P) - \frac{1}{2}\int \dist(x,y)^2\D\gamma(x,y)
	\end{align*}
	and the proof is complete.
	\qed
\end{proof}

%%%%%%%%%%%%%%%%%%%%%%%%%%%%%%%%%%%%%%%%%%%%%%%%%%%%%%%%%%%%%%%%%%%%%%%%%%%%%%%%%%%%%%%%%%%%%%%%%%%%%%%%%%%%%%%%%%%%%%%%
%%%%%%%%%%%%%%%%%%%%%%%%%%%%%%%%%%%%%%%%%%%%%%%%%%%%%%%%%%%%%%%%%%%%%%%%%%%%%%%%%%%%%%%%%%%%%%%%%%%%%%%%%%%%%%%%%%%%%%%%

\section{Convergence of $\Sch$ towards $\OT$}
\label{section:convergence_Sch_OT}
Here we give a new proof of the convergence of \emph{entropic} optimal transport towards \emph{deterministic} optimal transport as the diffusivity $\nu\to 0$, Theorem~\ref{thm:CV_Sch_OT}.
We stress again that the result itself is not new \cite{mikami2004monge,leonard2012schrodinger,cuturi2013sinkhorn,carlier2017convergence}, but our proof only relies on elementary PDE arguments and we believe it is worth including the details for the sake of completeness.

Before going into the proof we shall need a fundamental regularization procedure (Lemma~\ref{lem:fund} below), to be used repeatedly in the sequel.
To motivate the approach, observe that, in the previous section, the key step was the construction of a suitable recovery sequence $P^\nu\narrowcv P$ by means of Brownian bridges -- see in particular \eqref{eq:def_P^nu}.
Our regularization below will simply consist in a similar construction at the PDE level.

More precisely, recall that we write $\tau_s(x)$ for the heat kernel at time $s>0$
$$
\partial_s\tau_s=\frac 12\Delta\tau_s
$$
started from the initial Dirac distribution $\tau_0=\delta_0$.
For a given curve $\rho\in C([0,1];\P(\T^d))$ and diffusivity parameter $\nu>0$ we shall always write $\rho^\nu\in C([0,1];\P(\T^d))$ for the curve defined by
\begin{equation}
\label{eq:def_rho_nu_t}
t \quad \mapsto \quad \rho^\nu_t\coloneqq \rho_t * \tau_{\nu t(1-t)}\in\P(\T^d),
\end{equation}
where the convolution only acts in space.
In other words $\rho^\nu_t$ is defined as the solution of the heat flow at time $s=\nu t(1-t)$ started from $\rho_t$ at time $s=0$, and in particular $\rho^\nu$ has the same endpoints as $\rho$
$$
\rho^\nu_0=\rho_0 \qquad \mbox{and}\qquad \rho^\nu_1=\rho_1.
$$
Our regularity estimate takes the following quantitative form:
\begin{Lem}
 \label{lem:fund}
For all $\rho\in C([0,1];\P(\T^d))$, defining $\rho^\nu$ as in \eqref{eq:def_rho_nu_t}, there holds
 \begin{multline}
\label{eq:fund_weight}
\A(\rho^\nu)  +  \frac {\nu^2}{2} \int_0^1 \left(t -\frac 12\right)^2\int |\nabla\log\rho^\nu_t|^2 \D\rho^\nu_t \D t  + \nu \int_0^1 H(\rho^\nu_t)\D t
\\
\leq
\A(\rho) +\nu\frac{H(\rho_0)+H(\rho_1)}{2}.
\end{multline}
Moreover, there exists a dimensional constant $C=C_d>0$ such that,
 \begin{multline}
  \label{eq:fund_unweight}
  \A(\rho^\nu)  +  \frac {\nu^2}{8} \int_0^1 \hspace{-5pt} \int |\nabla\log\rho^\nu_t|^2 \D\rho^\nu_t \D t  + \nu \int_0^1 H(\rho^\nu_t)\D t
\\
\leq
\A(\rho) +\nu\left[
\frac{H(\rho_0)+H(\rho_1)}{2}
+C
\right].
 \end{multline}
 Similarly, for $\alpha>0$ the entropic version holds as:
 \begin{multline}
\label{eq:fund_weight_Halpha}
\H_\alpha(\rho^\nu)  +  \frac {\nu^2}{2} \int_0^1 \left(t -\frac 12\right)^2\int |\nabla\log\rho^\nu_t|^2 \D\rho^\nu_t \D t  + \nu \int_0^1 H(\rho^\nu_t)\D t
\\
\leq
\H_\alpha(\rho) +\nu\frac{H(\rho_0)+H(\rho_1)}{2}
\end{multline}
 and
 \begin{multline}
  \label{eq:fund_unweight_Halpha}
  \H_\alpha(\rho^\nu)  +  \frac {\nu^2}{8} \int_0^1 \hspace{-5pt} \int |\nabla\log\rho^\nu_t|^2 \D\rho^\nu_t \D t  + \nu \int_0^1 H(\rho^\nu_t)\D t
\\
\leq
\H_\alpha(\rho) +\nu\left[
\frac{H(\rho_0)+H(\rho_1)}{2}
+C
\right].
 \end{multline}
\end{Lem}
\noindent
Note that all four right-hand sides are allowed to be infinite, in which case our statement is vacuous.
\begin{proof}
Let us start with \eqref{eq:fund_weight}.
We can always assume that $\rho$ has regularity $AC^2([0,T];\P(\T^d))$, since otherwise $\A(\rho)=+\infty$ in the r.h.s.
 By theorem~\ref{thm:characterization_AC2} there exists a velocity field $c\in L^2(\D t\otimes \rho_t)$ such that $\A(\rho)=\frac 12\int_0^1\int |c_t|^2\rho_t\D t$.
 Defining the classical regularization
 \begin{equation}
 \label{eq:regul_momentum}
 \hat c^\nu_t\coloneqq \frac{(\rho_t c_t)*\tau_{\nu t(1-t)}}{\rho_t*\tau_{\nu t(1-t)}}=\frac{(\rho_t c_t)*\tau_{\nu t(1-t)}}{\rho^\nu_t},
 \end{equation}
 it is easy to check that
 $$
 \partial_t\rhont + \Div(\rhont \hat c^\nu_t) =-\nu \left(t-1/2\right)\Delta\rhont
 $$
 at least in the sense of distributions.
 (This is a first reason why we defined the regularization \eqref{eq:regul_momentum} as acting on the momentum variable $m^\nu=\rho^\nu_t\hat c^\nu_t=(\rho_t c_t)*\tau_{\nu t(1-t)}$, rather than directly on the velocities.)
 The extra Laplacian in the right-hand side arises because the regularizing kernel $\tau_{\nu t(1-t)}$ is not fixed but depends on time.
Setting moreover
\begin{equation}
\label{eq:def_cnu}
 c^\nu_t\coloneqq \hat c^\nu_t -\nu  \left(t-1/2\right)\nabla\log\rhont 
\end{equation}
and recalling that $\Div(\rho\nabla\log\rho)=\Div\left(\rho\frac{\nabla\rho}{\rho}\right)=\Delta\rho$, we have now
$$
 \partial_t\rhont + \Div(\rhont  c^\nu_t) =0,
 $$
 hence by definition of the kinetic action \eqref{eq:defA} and of the metric speed \eqref{eq:def_cmin_speed_MK}
 \begin{equation}
 \label{eq:action_rho_leq_cnu}
 \A(\rho^\nu)\leq \frac 12 \int _0^1 \hspace{-5pt} \int |c^\nu_t|^2\D\rho^\nu_t \D t.
 \end{equation}
 Since $\tau_{\nu t(1-t)}$ is a probability measure and $(\rho,m)\mapsto \frac{|m|^2}{2\rho}$ is jointly convex, an immediate application of Jensen's inequality gives
 \begin{equation}
 \label{eq:Jensen_action}
  \frac 12 \int_0^1 \hspace{-5pt}\int |\hat c^\nu_t|^2\D\rho^\nu_t \D t
 \leq \frac 12 \int_0^1 \hspace{-5pt} \int |c_t|^2 \D\rho_t \D t
 =\A(\rho)
 \end{equation}
 (This is another reason for the particular definition of $\hat{c}^\nu$.)
 Gathering \eqref{eq:action_rho_leq_cnu}\eqref{eq:Jensen_action} and exploiting \eqref{eq:def_cnu} to expand $|\hat c^\nu_t|^2$, we find
 \begin{multline}
 \label{eq:fund_weight_intermed}
\A(\rho^\nu)  +  \frac {\nu^2}{2} \int_0^1 \left(t -\frac 12\right)^2\int |\nabla\log\rho^\nu_t|^2 \D\rho^\nu_t \D t 
\\
\leq
\A(\rho) +\nu\int_0^1 \left(t -\frac 12\right)  \int \nabla\log\rhont\cdot c^\nu_t\,\D\rho^\nu_t\D t.
 \end{multline}
Writing next
\begin{equation*}
\frac{d}{d t}H(\rhont) = \int(1 + \log\rhont) \partial_t\rhont =  -\int(1 + \log\rhont) \Div(\rhont c^\nu_t)
=  \int \nabla\log \rhont\cdot \,c^\nu_t\,\D\rho^\nu_t\D t
\end{equation*}
we obtain after integration by parts
\begin{align*}
\int_0^1 \left(t -\frac 12\right) \int \nabla\log\rhont\cdot c^\nu_t\,\D\rho^\nu_t\D t
&= \int_0^1 \left(t -\frac 12\right) \frac{d}{dt}H(\rhont)\,\D t\\
&= \frac{H(\rho_0^\nu) + H(\rho_1^\nu)}{2}-\int_0^1 H(\rhont)\D t\\
&= \frac{H(\rho_0) + H(\rho_1)}{2}-\int_0^1 H(\rhont)\D t.
\end{align*}
Here we crucially used the fact that the endpoints $\rho^\nu_0=\rho_0$ and $\rho^\nu_1=\rho_1$ remain unchanged.
Our first estimate \eqref{eq:fund_weight} immediately follows by substituting this identity in \eqref{eq:fund_weight_intermed}.

To get \eqref{eq:fund_unweight}, we first add $\frac {\nu^2} 2\int_0^1 t(1-t)\int |\nabla\log\rhont|^2 \D\rho^\nu_t \D t$ to both sides of \eqref{eq:fund_weight} and use the algebraic identity $t(1-t)+(t-1/2)^2=1/4$ to get
\begin{multline}
 \label{eq:fund_unweight_intermed}
 \A(\rho^\nu)  +  \frac {\nu^2}{8} \int_0^1 \hspace{-5pt} \int |\nabla\log\rho^\nu_t|^2  \D\rho^\nu_t \D t + \nu \int_0^1 H(\rho^\nu_t)\D t \\
 \leq \A(\rho) +\nu\frac{H(\rho_0)+H(\rho_1)}{2} + \nu \int_0^1 \nu\frac{t(1-t)}{2}\int |\nabla\log\rhont|^2\, \D\rho^\nu_t \D t,
\end{multline}
and it only remains to control the last term in the right-hand side.
By the celebrated Li-Yau inequality \cite[Thm.~1.1]{li1986parabolic}, the Fisher information decays at a universal rate along the heat flow uniformly in the initial datum, here (with $\tilde\rho_s=\tau_s *\tilde\rho_0$):
$$
\int |\nabla\log \tilde\rho_s|^2 \D\tilde\rho_s \leq \frac d{2s}
\qquad \forall\, \tilde\rho_0\in\P(\T^d),\,\forall\,s>0.
$$
Recalling that, by definition, $\rhont$ is the solution at time $s=\nu t(1-t)$ of the heat flow started from $\tilde\rho_0=\rho_t$, the last term in \eqref{eq:fund_unweight_intermed} can thus be controlled as
\begin{equation*}
 \nu \int_0^1 \nu\frac{t(1-t)}{2}\int |\nabla\log\rhont|^2\, \D\rho^\nu_t \D t
 \leq \nu \int_0^1 \nu\frac{t(1-t)}{2}\cdot\frac{d}{2\nu t(1-t)} \D t
 \leq C_d\nu
\end{equation*}
and \eqref{eq:fund_unweight} follows.

As for \eqref{eq:fund_weight_Halpha}\eqref{eq:fund_unweight_Halpha}, we recall that the Fisher information
\begin{equation*}
F(\tilde\rho_s)=\frac 18\int |\nabla\log\tilde\rho_s|^2  \D\tilde\rho_s
\end{equation*}
is nonincreasing along the heat flow (by the same Jensen's inequality used in \eqref{eq:Jensen_action}).
In our particular setting this gives $F(\rhont)\leq F(\rho_t)$ for all $t\in[0,1]$.
The result immediately follows by  adding $\alpha^2\F(\rho^\nu)\leq \alpha^2\F(\rho)$ to \eqref{eq:fund_weight} and \eqref{eq:fund_unweight}, respectively, and the proof is complete.
\qed
\end{proof}

We are now in position of proving the convergence of $\Sch_\nu$ towards $\OT$.
\begin{proof}[Proof of Theorem~\ref{thm:CV_Sch_OT}]
The $\Gamma-\liminf$ is obvious, as $\H_\nu=\A+\nu^2\F\geq \A$ and $\A$ is lower semicontinuous.
Let us therefore consider the $\Gamma-\limsup$, and fix $\rho\in C([0,1];\P(\T^d))$ with endpoints $\rho_0,\rho_1$.
We can always assume that $\rho\in AC^2$, otherwise there is nothing to prove.

For $\rho\in AC^2([0,1];\P(\T^d))$, we claim that $(\rho^\nu)_{\nu>0}$ defined in \eqref{eq:def_rho_nu_t} is an admissible recovery sequence.
Indeed, as already discussed $\rho^\nu$ has same endpoints $\rho_0,\rho_1$ as $\rho$.
Moreover from \eqref{eq:fund_unweight} we get
$$
\H_\nu(\rho^\nu)=\A(\rho^\nu)+\nu^2\F(\rho^\nu)\leq \A(\rho) +\nu\left[
\frac{H(\rho_0)+H(\rho_1)}{2}
+C
\right],
$$
and taking the $\limsup$ gives
$$
\limsup\limits_{\nu\to 0}\,\H_\nu(\rho^\nu)
\leq \A(\rho)
$$
as required.
\qed
\end{proof}
\begin{Rem}
 From this proof it is clear that, apart from the static entropy term $\nu\frac{H(\rho_0)+H(\rho_1)}{2}$, the purely dynamical $\limsup$-gap for the $\Gamma$-convergence $\H_\nu\to\OT$ is of order at most $C\nu$ for a dimensional constant $C=C_d$.
\end{Rem}
\smallskip
\begin{Rem}
	\label{rem:marginal_Sch_OT}
	When $H(\rho_0)$ or $H(\rho_1)$ is infinite, the $\Gamma$-convergence \eqref{eq:Gamma_lim_Sch_OT} cannot hold because $\H_\nu + \iota_{(\rho_0,\rho_1)} \equiv + \infty$ for all $\nu>0$.
	However, similarly to the scenario of Theorem~\ref{thm:general_GammaCv_process}, it is still possible to prove
	\begin{equation*}
	\Gamma-\lim\limits_{\nu\to 0}\, \H_\nu 
	= \A ,
	\end{equation*}
	and more precisely
	\begin{equation*}
	\Gamma-\lim\limits_{\nu\to 0}\,\Big\{ \H_\nu + \iota_{(\rho^{\nu}_0,\rho^{\nu}_1)} \Big\}
	= \A + \iota_{(\rho_0,\rho_1)}
	\end{equation*}
	if and only if $\rho_0^\nu \narrowcv \rho_0$, $\rho_1^\nu \narrowcv \rho_1$ and
	\begin{equation*}
	\lim_{\nu \to 0} \nu H(\rho_0^{\nu}) = 0 \qquad \mbox{and} \qquad	\lim_{\nu \to 0} \nu H(\rho_1^{\nu}) = 0.
	\end{equation*}
	Such $(\rho_0^{\nu})$ and $(\rho_1^{\nu})$ are easy to build for instance by convolution.
\end{Rem}

%%%%%%%%%%%%%%%%%%%%%%%%%%%%%%%%%%%%%%%%%%%%%%%%%%%%%%%%%%%%%%%%%%%%%%%%%%%%%%%%%%%%%%%%%%%%%%%%%%%%%%%%%%%%%%%%%%%%%%%%
%%%%%%%%%%%%%%%%%%%%%%%%%%%%%%%%%%%%%%%%%%%%%%%%%%%%%%%%%%%%%%%%%%%%%%%%%%%%%%%%%%%%%%%%%%%%%%%%%%%%%%%%%%%%%%%%%%%%%%%%

\section{Convergence of $\MBro$ towards $\MREu$}
\label{section:convergence_MBro_MREu}
Here we prove the $\Gamma$-convergence in Theorem~\ref{thm:GammaCv_multiphase} as well as the convergence of the pressures associated with the incompressibility constraints, Theorem~\ref{thm:CV_pressures}.
%
%%%%%%%%%%%%%%%%%%%%%%%%%%%%%%%%%%%%%%%%%%%%%%%%%%%%%%%%%%%%%
\subsection{$\Gamma$-convergence}
\label{sec:gamma_CV_MBro_MREu}
\begin{proof}[Proof of Theorem~\ref{thm:GammaCv_multiphase}]
Again, the $\Gamma-\liminf$ easily follows from the standard lower semicontinuity of $\AA$ together with $\HH_\nu=\AA+\nu^2\FF\geq \AA$, and we only focus on the $\Gamma-\limsup$ inequality.
The argument essentially consists in superposing the proof of Theorem~\ref{thm:CV_Sch_OT} by linearity, \emph{i-e} integrating with respect to $\PP$.

More precisely: For $\nu>0$ we define the mapping
\begin{equation}
\label{eq:def_Phi_nu}
\Phi^\nu:\rho\mapsto \rho^\nu
\end{equation}
from $C([0,1];\P(\T^d))$ to itself, where the curve $\rho^\nu=(\rhont)_{t\in[0,1]}$ is defined in \eqref{eq:def_rho_nu_t}.
For any incompressible traffic plan $\PP$, we first claim that
\begin{equation*}
\PP^\nu\coloneqq {\Phi^\nu}\pf\PP
\end{equation*}
shares its marginals with $\PP$ and automatically inherits incompressibility from that of $\PP$.
Indeed, as already observed, $\rho^\nu$ leaves the endpoints unchanged $\rho^\nu_0=\rho_0$ and $\rho^\nu_1=\rho_1$, hence for all test functions $\varphi$ on $\P\times\P$
\begin{multline*}
 \int \varphi(\rho_0,\rho_1)\D \PP^\nu(\rho)
 =\int \varphi(\Phi^\nu(\rho)_0,\Phi^\nu(\rho)_1)\D \PP(\rho)\\
 =\int \varphi(\rho^\nu_0,\rho^\nu_1)\D \PP(\rho)
 =\int \varphi(\rho_0,\rho_1)\D \PP(\rho)
\end{multline*}
and therefore $\PP^\nu_{0,1}=\PP_{0,1}$.
In particular, the constraint $\PP^\nu_{0,1}=\Gamma$ is satisfied as soon as $\PP_{0,1}=\Gamma$.
For the incompressibility, since $\Leb$ is invariant for the heat flow $(\tau_s)\pf\Leb=\Leb\,\Rightarrow \, \Phi^\nu(\Leb)=\Leb$, and because $\Phi^\nu$ is linear, we have
\begin{equation*}
 \int \rho_t \D\PP^\nu(\rho)=
 \int \Phi^\nu(\rho)_t\D\PP(\rho) = \Phi^\nu\left(\int \rho \D\PP(\rho)\right)_t=\Phi^\nu(\Leb)=\Leb
\end{equation*}
for all $t$.

Taking now an admissible $\PP$ in $\MREu(\Gamma)$, we just showed that $\PP^\nu$ is admissible too, and we claim that it is a suitable recovery sequence.

First, we claim that $\PP^\nu\narrowcv \PP$: according to \cite[Lem.~5.2.1]{ambrosio2006gradient}, it suffices to show that $\Phi^{\nu}$ converges uniformly towards the identity on the compact sets of $C([0,1]; \P(\T^d))$.
In fact, this convergence is even uniform (and not compactly uniform), which will follow from the estimate
\begin{equation*}
\forall \rho \in \P(\T^d), \, \forall s\geq 0, \qquad \DMK(\rho, \rho * \tau_s) \leq C_d\sqrt{s}
\end{equation*}
for some dimensional constant $C_d>0$.
To check this, we use as an admissible coupling between $\rho$ and $\rho * \tau_s$ the joint law of the Brownian motion starting from $\rho$ between times $0$ and $s$: testing $\D\gamma_s(x,y)\coloneqq \tau_s(y-x)\D\rho(x)\otimes \D y$ as a competitor in \eqref{eq:def_DMK} yields
\begin{multline*}
 \DMK^2(\rho,\rho * \tau_s)\leq \frac 12\int \dist^2(x,y)\D\gamma_s(x,y)\\
 =\frac 12\int\underbrace{\left(\int \dist^2(x,y)\tau_s(x-y)\D y\right)}_{\leq C_d s \mbox{ by }\eqref{eq:estim_heat_flow}}\D \rho(x)
 \leq  C_d s.
\end{multline*}
As a consequence,
\begin{equation*}
\forall \rho \in C([0,1]; \P(\T^d)), \quad \sup_{t \in [0,1]} \DMK(\rho_t, \rho^{\nu}_t) \leq \sup_{t \in[0,1]} C_d\sqrt{\nu t(1-t)} = C_d\sqrt{\nu},
\end{equation*}
which proves the uniform convergence.

For the $\limsup$ inequality, we can always assume that $\AA(\PP)<\infty$ hence that $\PP$ only charges $AC^2$ curves, otherwise there is nothing to prove.
We can therefore appeal to Lemma~\ref{lem:fund} and \eqref{eq:fund_unweight} (for $\PP$-a.e. $\rho$) to estimate
\begin{align}
\notag \HH_\nu(\PP^\nu)
&=\int \H_\nu(\rho)\D\PP^\nu(\rho)
 =\int \H_\nu(\rho^\nu)\D\PP(\rho)\\
\notag & \leq \int \A(\rho)\D \PP(\rho) + \nu \int \left[
\frac{H(\rho_0)+H(\rho_1)}{2}
+C
\right]\D\PP(\rho)\\
\label{eq:estimate_HH_finite} &\leq \AA(\PP)+\nu\left[
\int\frac{H(\rho_0)+H(\rho_1)}{2}\D\Gamma(\rho_0,\rho_1)
 \, + C\right].
\end{align}
Taking the $\limsup$ gives the desired inequality and achieves the proof.
\qed
\end{proof}
\begin{Rem}
	As in Remark~\ref{rem:marginal_Sch_OT}, if \eqref{eq:M_initial_final_entropy_finite} does not hold the $\Gamma$-convergence \eqref{eq:Gamma_lim_MBro_MREu} cannot hold due to $\HH_\nu+ \boldsymbol\iota_{\Gamma} \equiv + \infty$.
	However, it is still possible to prove
	\begin{equation*}
	\Gamma-\lim\limits_{\nu\to 0}\, \HH_\nu 
	= \AA
	\end{equation*}
	by regularizing $\Gamma$ as $\Gamma^{\nu} \coloneqq  (\Psi_{\nu}, \Psi_{\nu})\pf \Gamma$, where
	\begin{equation*}
	\Psi_{\nu}: \rho \in \P(\T^d) \mapsto \rho * \tau_{\nu} \in \P(\T^d) 
	\end{equation*}
	and the admissible generalized flow must also be regularized correspondingly.
	In fact, we expect as in Remark~\ref{rem:marginal_Sch_OT} that
	\begin{equation*}
	\Gamma-\lim\limits_{\nu\to 0}\,\Big\{ \HH_\nu + \boldsymbol\iota_{\Gamma^{\nu}} \Big\}
	= \AA + \boldsymbol\iota_{\Gamma}
	\end{equation*}
	if and only if $\Gamma^{\nu} \narrowcv \Gamma$ and
	\begin{equation*}
	\lim_{\nu \to 0} \nu \int H(\rho_0) \D \Gamma^{\nu}(\rho_0, \rho_1) = 0 \quad \mbox{and} \quad	\lim_{\nu \to 0} \nu \int H(\rho_1) \D \Gamma^{\nu}(\rho_0, \rho_1) = 0.
	\end{equation*}
	To prove this statement, one needs a result corresponding to \cite[Thm.~1]{baradat2018continuous} in that setting but we did not pursue in this direction.
\end{Rem}
%%%%%%%%%%%%%%%%%%%%%%%%%%%%%%%%%%%%%%%%%%%%%%%%%%%%%%%%

%%%%%%%%%%%%%%%%%%%%%%%%%%%%%%%%%%%%%%%%%%%%%%%%%%%%%%%%%%%%%%%%%%%%%%%%%%%%%%%%%%%%%%%%%%%%%%%%%%%%%%%%%%%%%%%%%%%%%%%%
%%%%%%%%%%%%%%%%%%%%%%%%%%%%%%%%%%%%%%%%%%%%%%%%%%%%%%%%%%%%%%%%%%%%%%%%%%%%%%%%%%%%%%%%%%%%%%%%%%%%%%%%%%%%%%%%%%%%%%%%

\subsection{Convergence of the pressures}
\label{subsec:pressure}
Let us now turn to the proof of Theorem~\ref{thm:CV_pressures}.
We will need the following definition of the average density of a traffic plan.
\begin{Def}
	Let $\PP$ be a traffic plan. Its density $\rho^{\PP} \in C([0,1]; \P(\T^d))$ is defined at time $t \in [0,1]$ by:
	\begin{equation*}
	\rho^{\PP}_t \coloneqq  \int \rho_t \D \PP(\rho).
	\end{equation*}
\end{Def}
In other words, $(\rho^\PP_t)_{t\in[0,1]}$ is the curve obtained by averaging all the phases at time $t$ with respect to $\PP$.
The key ingredient below will rely on the proof of Theorem 8.4 in \cite{baradat2018existence}.
There, the first author introduced the following functional space:
\begin{Def}
	We define $\mathcal{E}_0$ the space of continuous functions $f:[0,1] \times \T^d \to \R$ satisfying:
	\begin{itemize}
		\item for all $x \in \T^d$, $f(0,x) = f(1,x) = 0$, 
				\item for all $t \in [0,1]$, 
		\begin{equation}
		\label{eq:zeromean}
		\int_{\T^d} f(t,x) \D x = 0,
		\end{equation}
		\item for all $t \in [0,1]$, $f(t,\bullet) \in W^{2, \infty}(\T^d)$ and
		\begin{equation*}
		\sup_{t \in [0,1]} \| \dd^2 f(t,\bullet) \|_{\infty} < + \infty,
		\end{equation*}
		\item 
		$f(\bullet, x) \in AC^2([0,1])$ for all $x \in \T^d$, and $\partial_t f$, which is well defined for almost all $t$ and all $x$, satisfies
		\begin{equation*}
		\int_0^1 \| \partial_t f(t,  \bullet) \|_{\infty}^2 \D t < + \infty.
		\end{equation*}
	\end{itemize}
We endow $\mathcal{E}_0$ with the norm
\begin{equation*}
\|f\|_{\mathcal E _0} \coloneqq  \sup_{t \in [0,1]} \| \dd^2 f(t,\bullet) \|_{\infty} + \left(\int_0^1 \| \partial_t f(t,\bullet) \|_{\infty}^2 \D t \right)^{1/2},
\end{equation*}
for which $(\mathcal{E}_0, {\|\bullet\|_{\mathcal E_0}})$ is a Banach space.
We write $\mathcal{E}'_0$ for its topological dual and $\|\bullet\|_{\mathcal E_0'}$ for the dual norm.
Note that $\mathcal E_0'\subset \mathcal D'((0,1)\times\T^d)$ is a subspace of distributions.
\end{Def}
Taking into account the dependence on the diffusivity parameter $\nu$, an easy extension of \cite[Thm.~8.4]{baradat2018existence} gives 
\begin{Thm}[Existence of the pressure fields]
\label{thm:existence_pressure}
~
\begin{enumerate}
 \item 
 \label{item:pressure_REUl}
	Let $\Gamma$ be bistochastic in average.
	There exists a unique $p \in \mathcal{E}_0'$ such that, for all solutions $\PP$ to $\MREu(\Gamma)$ and all traffic plans $\QQ$ satisfying \eqref{eq:Mjoint_law} with $\rho^{\QQ} - 1 \in \mathcal{E}_0$, there holds
	\begin{equation*}
	\AA(\QQ) \geq \AA(\PP) + \cg p, \rho^{\QQ}-1 \cd_{\mathcal{E}_0', \mathcal{E}_0}.
	\end{equation*}
	The distribution $p$ is called the pressure field associated with $\MREu(\Gamma)$.
\item
	Let $\Gamma$ be bistochastic in average and satisfy \eqref{eq:M_initial_final_entropy_finite}, and let $\nu>0$.
	There exists a unique $p^{\nu} \in \mathcal{E}_0'$ such that, if $\PP^{\nu}$ is the unique solution to $\MBro_{\nu}(\Gamma)$ and $\QQ$ is any traffic plan satisfying \eqref{eq:Mjoint_law} with $\rho^{\QQ} - 1 \in \mathcal{E}_0$, then
	\begin{equation}
	\label{eq:pressure_nu_ML}
	\HH_{\nu}(\QQ) \geq \HH_{\nu}(\PP^{\nu}) + \cg p^{\nu}, \rho^{\QQ}-1 \cd_{\mathcal{E}_0', \mathcal{E}_0}.
	\end{equation}
	Moreover, there exists a dimensional constant $C$ such that
	\begin{equation}
	\label{eq:estimate_pressure}
	\left\|p^{\nu}\right\|_{\mathcal E_0'} \leq C(1 + \nu^2).
	\end{equation}
	The distribution $p^{\nu}$ is called the pressure field associated with $\MBro_{\nu}(\Gamma)$.
\end{enumerate}
\end{Thm}
\begin{proof}
	In the absence of viscosity, \emph{i-e} point~\ref{item:pressure_REUl} in our statement, the result is due to Brenier in \cite{brenier1993dual} but also follows from our argument below.
	In the viscous setting only our estimate \eqref{eq:estimate_pressure} is new compared to \cite{baradat2018existence} hence we only sketch the proof for $\nu>0$ and refer to \cite{baradat2018existence} for more details.

	Given a scalar function $\varphi = (\varphi(t,x))$, we say that a traffic plan $\QQ$ is admissible for $\MBro_{\nu}^{\varphi}(\Gamma)$ if \eqref{eq:Mjoint_law}\eqref{eq:Mfinite_action} hold, the previous incompressibility \eqref{eq:incomressibility_MREu} is replaced by
	\begin{equation*}
	\rho^{\QQ}_t = (1 + \varphi(t, \bullet)) \Leb,
	\end{equation*}
	and the Fisher information \eqref{eq:Mfinite_fischer} is finite.
	Of course, this is possible only if $1 + \varphi \geq 0$ and $\varphi$ satisfies \eqref{eq:zeromean}, but both properties hold if $\varphi$ is chosen sufficiently small in $\mathcal{E}_0$.
	
	Note that $\MBro_\nu$ corresponds to $\varphi\equiv 0$: essentially, as explained in the introduction, the idea of proof below will consist in relaxing the incompressibility constraint and performing a small perturbation of $\MBro_\nu=\MBro_\nu^{|\varphi=0}$ around $\varphi=0$.
	
	With the convention that $\inf\limits_\emptyset \{\dots\}= + \infty$ and following \cite{baradat2018existence}, the map
	\begin{equation*}
	\phi^{\nu}: 
	\quad\varphi \in \mathcal{E}_0 \quad\longmapsto \quad\inf_{
	\substack{
	\QQ \mbox{ \small admissible}\\
	\mbox{\small for } \MBro_{\nu}^{\varphi}(\Gamma)
	}
	}  \HH_{\nu}(\QQ)
	\end{equation*}
	is convex and lower semicontinuous (for the topology induced by the norm $\| \bullet\|_{\mathcal E_0}$).
	Hence, in order to get the existence of a subgradient $p^{\nu}\in \partial_{\varphi=0} \phi^\nu$ satisfying \eqref{eq:pressure_nu_ML} and \eqref{eq:estimate_pressure}, it suffices to show that there exists a constant $C=C_d>0$ such that for all $\varphi \in \mathcal{E}_0$ close to $0$ -- say with $\|\varphi\|_{\mathcal E_0} \leq 1/2$ -- there holds
	\[
	\phi^{\nu}(\varphi) \leq C(1 + \nu^2).
	\]
	But for such a $\varphi$, using a famous result by Dacorogna and Moser (see \cite[Thm. 8.1]{baradat2018existence} or \cite{dacorogna1990partial} for the original version) and arguing as in \cite{baradat2018existence}, one can build from the minimizer $\PP^{\nu}$ in $\MBro_\nu$ a traffic plan $\QQ^{\nu}$ that is admissible for $\MBro_{\nu}^{\varphi}(\Gamma)$, with moreover
	\begin{gather*}
	\AA(\QQ^{\nu}) \leq C(1 + \AA(\PP^{\nu})),\\
	\FF(\QQ^{\nu}) \leq C(1 + \FF(\PP^{\nu})),
	\end{gather*}
	and $C$ depending only on the dimension.
	In particular, we have
	\begin{equation*}
	\phi^{\nu}(\varphi) \leq \HH_{\nu}(\QQ^{\nu}) = \AA(\QQ^{\nu}) + \nu^2 \FF(\QQ^{\nu}) \leq C(1 + \nu^2 + \HH_{\nu}(\PP^{\nu})).
	\end{equation*}
	We conclude using the fact that $\PP^\nu$ is a minimizer, hence
	\begin{equation}
	\label{eq:estimate_Hnu_Pnu_leq_C}
	\HH_{\nu}(\PP^{\nu}) \leq \HH_\nu(\PP^1)=\AA(\PP^1) + \nu^2 \FF(\PP^1) \leq C(1 + \nu^2)
	\end{equation}
	(Here $\PP^1$ is the minimizer of $\MBro_1$ with $\nu=1$, but any other fixed admissible traffic plan would have worked as well.)
	We refer to \cite[Thm.~1.2]{brenier1993dual} and \cite[Thm.~8.4]{baradat2018existence} for the uniqueness part of the statement.
	\qed
\end{proof}
We are now ready to prove the convergence of the pressures:
\begin{proof}[Proof of Theorem~\ref{thm:CV_pressures}]
	Let $\PP^{\nu}$ be the unique minimizer for $\MBro_{\nu}(\Gamma)$.
	From \eqref{eq:estimate_Hnu_Pnu_leq_C} we have
	$$
	\AA(\PP^\nu)\leq \AA(\PP^\nu)+\nu^2\FF(\PP^\nu)= \HH_{\nu}(\PP^\nu)\leq  C(1+\nu^2)\leq C
	$$
	as $\nu\to 0$,	and we recall that $\AA$ is proper for the narrow topology. 
	Consequently, the sequence $(\PP^{\nu})$ is tight and has at least one cluster point $\PP^\nu\narrowcv \PP$ (up to extraction of a discrete subsequence if needed).
	By Theorem~\ref{thm:GammaCv_multiphase}, and recalling that $\Gamma$-convergence implies convergence of minimizers to minimizers, it is clear that any such cluster point $\PP$ is a solution to $\REu(\Gamma)$.
	
	First of all, by Theorem~\ref{thm:existence_pressure}, the pressures $(p^{\nu})_{\nu>0}$ are bounded in $\mathcal{E}_0'$ uniformly in $\nu$.
	But, as the functional space $\mathcal{G}$ from Definition~\ref{defi:functional_space_G} is continuously embedded in $\mathcal{E}_0$, $(p^{\nu})$ is also bounded in $\mathcal{G}'$.
	By separability of $\mathcal{G}$ and the Banach-Alaoglu theorem there is $p^* \in \mathcal{G}'$ such that, up to extraction of a further subsequence, $p^{\nu}{\narrowcv} p^*$ for the weak-$*$ topology of $\mathcal G'$.
	Thus it suffices to show that $p^* = p$, and convergence of the whole sequence towards $p$ will follow by standard  arguments.
	To do so, and by uniqueness in Theorem~\ref{thm:existence_pressure}, it suffices to show that, for all $\varphi \in \mathcal{E}_0$ and all traffic plans $\QQ$ satisfying \eqref{eq:Mjoint_law}\eqref{eq:Mfinite_action} as well as $\rho^{\QQ} = 1 + \varphi$, there holds
	\begin{equation}
	\label{eq:p*_pressure}
	\AA(\QQ) \geq \AA(\PP) + \cg p^*, \varphi \cd_{\mathcal E_0',\mathcal E_0}.
	\end{equation}
	To test this inequality, take $\varphi \in \mathcal{E}_0$ and $\QQ$ such that $\rho^{\QQ} = 1 + \varphi$, and define
	\begin{equation*}
	\QQ^{\nu} \coloneqq  \Phi^{\nu}\pf \QQ
	\end{equation*}
	as in the proof of Theorem~\ref{thm:GammaCv_multiphase}.
	We recall that $\Phi^\nu$ is defined in \eqref{eq:def_Phi_nu} \emph{via} \eqref{eq:def_rho_nu_t}.
	From the subdifferential characterization \eqref{eq:pressure_nu_ML} of $p^\nu$, we have
	\begin{equation}
	\label{eq:pnu_ML}
	\HH_{\nu}(\QQ^{\nu}) \geq \HH_{\nu}(\PP^{\nu}) + \cg p^{\nu}, \rho^{\QQ^{\nu}} - 1 \cd_{\mathcal E_0',\mathcal E_0}.
	\end{equation}
	Repeating the exact same argument from the proof of Theorem~\ref{thm:GammaCv_multiphase} (construction of the recovery sequences), we have moreover
	\begin{equation}
	\label{eq:Qnu_recovery}
	\AA(\QQ) \geq \limsup_{\nu \to 0} \HH_{\nu}(\QQ^{\nu}) .
	\end{equation}
	But by linearity of $\Phi^\nu$ with $\Phi^\nu(\Leb)=\Leb$ (or, abusing notations, $\Phi^\nu(1)=1$) it is easy to check that
	\begin{equation*}
	\varphi^{\nu} \coloneqq  \rho^{\QQ^{\nu}} - 1 =\Phi^{\nu}(\varphi),
	\end{equation*}
	which by definition of $\Phi^\nu$ simply means that $\varphi^\nu(t,\bullet)$ is the solution at time $s=\nu t(1-t)$ of the heat flow started from $\varphi(t,\bullet)$. 
	Therefore, by standard properties of the heat flow, $\varphi^{\nu} \to \varphi$ in any reasonable topology, and in particular strongly in $\mathcal E_0$.
	Together with $p^{\nu} \narrowcv p$, this allows to take the limit in the product
	\begin{equation}
	\label{eq:cv_weak_strong}
	\cg p^{\nu}, \varphi^{\nu} \cd_{\mathcal E_0',\mathcal E_0}\xrightarrow[\nu \to 0]{} \cg p^*, \varphi \cd_{\mathcal E_0',\mathcal E_0}.
	\end{equation}
	Moreover, by the $\Gamma-\liminf$ property in Theorem~\ref{thm:GammaCv_multiphase} with $\PP^\nu\narrowcv \PP$, 
	\begin{equation}
	\label{eq:Gamma_liminf_minimizers}
	\liminf_{\nu \to 0} \HH_{\nu}(\PP^{\nu}) \geq \AA(\PP).
	\end{equation}
	We finally retrieve \eqref{eq:p*_pressure} by passing to the limit in \eqref{eq:pnu_ML} using \eqref{eq:Qnu_recovery}, \eqref{eq:cv_weak_strong}, and \eqref{eq:Gamma_liminf_minimizers}.
	\qed
\end{proof}

%%%%%%%%%%%%%%%%%%%%%%%%%%%%%%%%%%%%%%%%%%%%%%%%%%%%%%%%%%%%%%%%%%%%%%%%%%%%%%%%%%%%%%%%%%%%%%%%%%%%%%%%%%%%%%%%%%%%%%%%
%%%%%%%%%%%%%%%%%%%%%%%%%%%%%%%%%%%%%%%%%%%%%%%%%%%%%%%%%%%%%%%%%%%%%%%%%%%%%%%%%%%%%%%%%%%%%%%%%%%%%%%%%%%%%%%%%%%%%%%%

\section{Time-convexity of the entropy}
\label{section:convexity}
Using our regularization lemma~\ref{lem:fund}, we can now give the proofs of Proposition~\ref{prop:convexity_OT_Sch} and Proposition~\ref{prop:convexity_MREu_MBro}.
We begin with the single-phase setting:
\begin{proof}[Proof of Proposition~\ref{prop:convexity_OT_Sch}]
 Let us start with $\OT$.
 For small $\nu>0$ consider the curve $\rho^\nu$ defined by \eqref{eq:def_rho_nu_t}.
 As already discussed the endpoints remain invariant, $\rho^\nu_0=\rho_0,\rho^\nu_1=\rho_1$: the curve $\rho^\nu$ is therefore an admissible competitor in the $\OT$ problem, and since $\rho$ is a minimizer we have $\A(\rho)\leq \A(\rho^\nu)$.
Simply discarding the term $\frac{\nu^2}8\int (\dots)\geq 0$ in \eqref{eq:fund_weight}, we have
 \begin{equation*}
\A(\rho)\leq \A(\rho^\nu)\leq \A(\rho) +  \nu
\left[   \frac{H(\rho_0)+H(\rho_1)}{2}-\int_0^1H(\rhont)\D t\right]
 \end{equation*}
hence
\begin{equation*}
 \int_0^1H(\rhont)\D t  \leq  \frac{H(\rho_0)+H(\rho_1)}{2}
\end{equation*}
for all $\nu>0$.
By standard properties of the heat flow we have moreover $\rhont=\tau_{\nu t(1-t)}*\rho_t\narrowcv \rho_t$ for all $t\in[0,1]$ as $\nu\to 0$.
Since the entropy is lower semicontinuous with respect to narrow convergence, we get by Fatou's lemma
\begin{equation*}
 \int_0^1H(\rho_t)\D t 
 \leq 
 \int_0^1\liminf\limits_{\nu\to 0} H(\rhont)\D t
 \leq  \liminf\limits_{\nu\to 0}\int H(\rhont)\D t
 \leq \frac{H(\rho_0)+H(\rho_1)}{2}.
\end{equation*}
In particular $H(\rho_t)<\infty$ for a.e. $t$, and in fact for all $t$ by narrow continuity of $\rho\in C([0,1];\P(\T^d))$ and lower semicontinuity of $H$.

This was carried out in times $t\in [0,1]$, but $\rho$ is of course a minimizer for the optimal transport problem $\OT(\rho_{t_0},\rho_{t_1})$ for all intermediate times $t_0\leq t_1$.
Since we just proved that $\frac 12[H(\rho_{t_0})+H(\rho_{t_1})]<+\infty$, we can repeat the exact same argument to conclude that
$$
\int_{t_0}^{t_1}H(\rho_t)\D t \leq 
\frac{H(\rho_{t_0})+H(\rho_{t_1})}{2},
\qquad \forall\, t_0\leq t_1.
$$
Since $\rho$ is narrowly continuous and $H$ is l.s.c. for the narrow convergence, and because $t_0,t_1$ can now vary arbitrarily, this implies the desired convexity.

The proof for $\Sch_\alpha(\rho_0,\rho_1)$ is identical, simply using \eqref{eq:fund_weight_Halpha} instead of \eqref{eq:fund_weight}.
\qed
\end{proof}

For the multiphase setting the proof essentially consists in superposing the previous argument by linearity:
\begin{proof}[Proof of Proposition~\ref{prop:convexity_MREu_MBro}]
We consider first the $\MREu$ problem.
As in the proof of Theorem~\ref{thm:GammaCv_multiphase} earlier in section~\ref{sec:gamma_CV_MBro_MREu}, the argument essentially consists in superposing (\emph{i.e} integrating with respect to $\PP$) the corresponding statement for a single phase, here Proposition~\ref{prop:convexity_OT_Sch}.

More precisely: let $\PP$ be a solution to $\MREu$, and consider as before the map ${\Phi^\nu:\rho \mapsto \rho^\nu}$ from $C([0,1];\P(\T^d))$ to itself defined by \eqref{eq:def_rho_nu_t}.
We already checked in the proof of Theorem~\ref{thm:GammaCv_multiphase} that the traffic plan
\begin{equation*}
\PP^\nu\coloneqq {\Phi^\nu}\pf\PP
\end{equation*}
is incompressible and shares its marginals $\Gamma$ with $\PP$.
Since $\PP$ is a minimizer in $\MREu$ there holds
\begin{equation*}
\AA(\PP)\leq \AA(\PP^\nu) =\int \A(\rho)\D\PP^\nu(\rho) =\int \A\circ\Phi^\nu(\rho)\D\PP(\rho) =\int \A(\rho^\nu)\D\PP(\rho).
\end{equation*}
Discarding $\frac{\nu^2}{8}\int(\dots)\geq 0$ in \eqref{eq:fund_weight} we can estimate as before
\begin{equation*}
 \A(\rho^\nu)\leq \A(\rho) +\nu\left(\frac{H(\rho_0)+H(\rho_1)}{2}-\int_0^1H(\rho^\nu_t)\D t\right)
\end{equation*}
for $\PP$-a.e. $\rho$, and integrating with respect to $\PP$ thus gives
\begin{align*}
 \AA(\PP)
 &\leq \int \left\{
 \A(\rho) +\nu\left(\frac{H(\rho_0)+H(\rho_1)}{2}-\int_0^1H(\rho^\nu_t)\D t\right)
 \right\}\D\PP(\rho)\\
  &= \AA(\PP) +\nu \left(\int\frac{H(\rho_0)+H(\rho_1)}{2} \D\PP(\rho)-\int \hspace{-5pt} \int_0^1 H(\rho^\nu_t)\D t \D\PP(\rho)\right).
\end{align*}
Whence
\begin{equation*}
 \int_0^1 \hspace{-5pt} \int H(\rho^\nu_t)\D\PP(\rho) \D t 
 \leq
 \int\frac{H(\rho_0)+H(\rho_1)}{2} \D\PP(\rho)
\end{equation*}
for all $\nu>0$.
The right-hand side is finite since the marginal entropies $\int \{H(\rho_0)+H(\rho_1)\}\D\PP(\rho)=\int \{H(\rho_0)+H(\rho_1)\}\D\Gamma(\rho_0,\rho_1)<+\infty$.
Taking first $\nu\to 0$ and repeating next the argument in arbitrary subintervals $[t_0,t_1]\subset[0,1]$, the rest of the proof is identical to the previous proof of Proposition~\ref{prop:convexity_OT_Sch} and we omit the details.

For $\MBro_\alpha$ we simply use \eqref{eq:fund_weight_Halpha} instead of \eqref{eq:fund_weight} as before, and the proof is complete.
\qed
\end{proof}
\begin{Rem}
In \cite{lavenant2017time} H. Lavenant proves (a slightly weaker version of) the same convexity by discretizing $\MREu$ in time, which gives a minimization problem over a large number $K$ of intermediate marginals at times $0=t_0,\dots, t_K=1$.
Performing an infinitesimal perturbation of the $k$-th optimal marginal using the heat flow as well as the \emph{flow interchange} technique from \cite{matthes2009family}, one retrieves then some convexity in the discrete time variable $k$ and finally passes to the limit $K\to\infty$ to conclude.
The technical details differ compared to our proof above, but the main idea is somehow similar: the heat flow gives admissible competitors in the variational problem, and tends to simultaneously diminish and convexify the entropy.
Hence if the entropy were not convex, one could construct better competitors by running the heat flow for short times while improving convexity.
However, our regularization $\rhont=\tau_{\nu t(1-t)}*\rho_t$ is more global,
roughly speaking because we simultaneously perturb the whole continuum of time marginals in a unified fashion and thus we avoid any delicate time-slicing.
More importantly, our approach has a clear counterpart at the level of the underlying stochastic processes, $\nu t(1-t)$ being of course the intrinsic scale of the Brownian bridges $B^{\nu,x,y}$ involved in section~\ref{section:GammaCv_process}.
\end{Rem}
%%%%%%%%%%%%%%%%%%%%%%%%%%%%%%%%%%%%%%%%%%%%%%%%%%%%%%%%%%%%%%%%%%%%%%%%%%%%%%%%%%%%%%%%%%%%%%%%%%%%%%%%%%%%%%%%%%%%%%%%
%%%%%%%%%%%%%%%%%%%%%%%%%%%%%%%%%%%%%%%%%%%%%%%%%%%%%%%%%%%%%%%%%%%%%%%%%%%%%%%%%%%%%%%%%%%%%%%%%%%%%%%%%%%%%%%%%%%%%%%%
\subsection*{Acknowledgements}
\noindent
This work is part of A.B.'s PhD thesis supervised by Y. Brenier and D. Han-Kwan, and he would like to thank both of them for their support.
It was initiated during a visit in GFM University of Lisbon, and A.B. is grateful to A.B. Cruzeiro and J.C. Zambrini for making this visit possible.
L.M. also wishes to thank A.B. Cruzeiro and J.C. Zambrini for useful discussions and their kind support.
Both authors want to express their gratitude to C. L\'eonard for his continued enthusiasm when discussing the Schr\"odinger problem.
% \end{acknowledgements}

%%%%%%%%%%%%%%%%%%%%%%%%%%%%%%%%%%%%%%%%%%%%%%%%%%%%%%%%%%%%%%%%%%%%%%%%%%%%%%%%%%%%%%%%%%%%%%%%%%%%%%%%%%%%%%%%%%%%%%%%
%%%%%%%%%%%%%%%%%%%%%%%%%%%%%%%%%%%%%%%%%%%%%%%%%%%%%%%%%%%%%%%%%%%%%%%%%%%%%%%%%%%%%%%%%%%%%%%%%%%%%%%%%%%%%%%%%%%%%%%%
\begin{appendices}

%
%%%%%%%%%%%%%%%%%%%%%%%%%%%%%%%%%%%%%%%%%%%%%%%%%%%%%%%%%%%%%%%%%%%%%%%%%%%%%%%%%%%%%%%%%%%%%%%%%%%%%%%%%%%%%%%%%%%%%%%%
\section{Existence and uniqueness for $\MBro$}
	\label{sec:exist_unique_MBro}
Here we establish
\begin{Thm}
\label{thm:exist_unique_MBro}
	Take $\Gamma$ bistochastic in average satisfying the entropy condition \eqref{eq:M_initial_final_entropy_finite}, and let $\nu >0$.
	Then $\MBro_{\nu}(\Gamma)$ admits a unique solution.
\end{Thm}
This is perhaps not completely standard in this form due to our choice of exposition in terms of traffic plans, and we include the details for the sake of completeness.
\begin{proof}
	For the existence it suffices to show that there exists at least one admissible traffic plan. ($\HH_{\nu}$ being proper and lower semicontinuous, the direct method in the calculus of variations applies.)
	In order to find such a traffic plan, one can either adapt the proof of \cite[Cor.~5.2]{arnaudon2017entropic}, or also observe that, given an admissible traffic plan $\PP$ for $\MREu(\Gamma)$, the traffic plan $\PP^{\nu}=\Phi^\nu\pf \PP$ constructed in the proof of Theorem~\ref{thm:GammaCv_multiphase} is admissible for $\MBro_{\nu}(\Gamma)$ -- in particular \eqref{eq:estimate_HH_finite} ensures that $\HH_\nu(\PP^\nu)< +\infty$.
	
	For the uniqueness part, we first show that if $\PP$ is a solution to $\MBro_{\nu}(\Gamma)$, then the conditional law $\PP^{\rho_0, \rho_1} \coloneqq  \PP(\,\bullet\, | \XX_0 = \rho_0 , \XX_1 = \rho_1)$ is a Dirac mass for $\Gamma$-almost all $(\rho_0, \rho_1)$.
	In other words, $\PP$ is supported on the graph of a measurable map, which assigns to any $(\rho_0, \rho_1)$ a unique curve $m=m[\rho_0,\rho_1] \in C([0,1]; \P(\T^d))$ joining $\rho_0$ to $\rho_1$.
	Indeed, let us define the average:
	\begin{equation*}
	m[\rho_0, \rho_1] \coloneqq  \int \rho \D \PP^{\rho_0, \rho_1}(\rho) \quad\in C([0,1];\P(\T^d)).
	\end{equation*}
	This curve is well defined for $\Gamma$-almost all $(\rho_0, \rho_1)$,
	and we claim that $\PP^{\rho_0, \rho_1} = \boldsymbol\delta_{m[\rho_0, \rho_1]}$ for $\Gamma$-almost all $(\rho_0, \rho_1)$.
	To check this, let us define
	\begin{equation*}
	\widetilde{\PP} \coloneqq  \int \boldsymbol\delta_{m[\rho_0, \rho_1]} \D \Gamma(\rho_0, \rho_1).
	\end{equation*}
	Because $m[\rho_0,\rho_1]$ has endpoints $\rho_0,\rho_1$ one can check that $\widetilde{\PP}_{0,1} = \Gamma$, and in the same spirit it is easy to see that $\widetilde\PP$ is incompressible in average (because $\PP$ is).
	By strict convexity of $\H_{\nu}$ and Jensen's inequality, we have for $\Gamma$-almost all $(\rho_0, \rho_1)$ 
	\begin{equation}
	\label{eq:Jensen_PPtilde}
	\H_{\nu}(m[\rho_0, \rho_1])
	=\H_\nu\left(\int \rho \D \PP^{\rho_0, \rho_1}(\rho)\right)
	\leq \int \H_{\nu}(\rho) \D \PP^{\rho_0, \rho_1}(\rho)
	\end{equation}
	with equality if and only if $\PP^{\rho_0, \rho_1}$ is a Dirac mass.
	To verify that equality holds as desired, let us integrate \eqref{eq:Jensen_PPtilde} with respect to $\Gamma$:
	by definition of $\widetilde{\PP}$ on the left-hand side, and using the disintegration formula \eqref{eq:disintegration} with respect to $\PP$ on the right-hand side,
	we get
	\begin{align*}
	 \HH_\nu(\widetilde \PP) &= \int \H_\nu(\rho)\D \widetilde\PP(\rho)=\iint \H_\nu(\rho)\D\boldsymbol \delta_{m[\rho_0,\rho_1]}(\rho)\D\Gamma(\rho_0,\rho_1)\\
	& =\int \H_\nu(m[\rho_0,\rho_1]) \D\Gamma(\rho_0,\rho_1)
	\\
	&\overset{\eqref{eq:Jensen_PPtilde}}{\leq} \int \left(\int\H_{\nu}(\rho) \D \PP^{\rho_0, \rho_1}(\rho)\right) \D\Gamma(\rho_0,\rho_1)\\
	 &= \int \H_\nu(\rho)\D\PP(\rho)=\HH_\nu(\PP).
	\end{align*}
	Since $\PP$ is a minimizer and $\widetilde\PP$ is admissible the reverse inequality $\HH_\nu(\PP)\leq \HH_\nu(\widetilde\PP)$ holds as well, thus we must have equality in \eqref{eq:Jensen_PPtilde} for $\Gamma$-a.e. $(\rho_0,\rho_1)$ and therefore $\PP^{\rho_0, \rho_1} = \boldsymbol\delta_{m[\rho_0, \rho_1]}$ as claimed.
	
	Finally, if $\PP_1$ and $\PP_2$ are two solutions to $\MBro(\Gamma)$ then, because $\FF_{\nu}$ is affine, $\PP_3 \coloneqq  (\PP_1 + \PP_2)/2$ is a solution as well and must be supported on a graph.
	But, $\PP_1,\PP_2$ being themselves supported on a graph, $\PP_1+\PP_2$ can be supported on a graph if and only if $\PP_1$ and $\PP_2$ coincide.
	Hence, uniqueness is proved.
	\qed
\end{proof}
\begin{Rem}
From this proof it is clear that we established a slightly stronger statement, namely that any minimizer for $\MBro$ must be supported on a graph $(\rho_0,\rho_1)\mapsto \delta_{m[\rho_0,\rho_1]}$.
This shows that the framework of traffic plans is not much more general than multiphase flows in the sense of Brenier's parametric setting, \emph{i-e} when the phases $\rho=(\rho^a)_{a}$ are labeled using the initial and final positions $a=(x,y)\in\T^d\times \T^d$ and the incompressibility reads $\int_{\T^d\times\T^d}\rho^a_t \D a=\Leb$ for all $t$.
Somehow we just proved that one can allow labeling on couples $(\rho_0,\rho_1)\in\P(\T^d)\times \P(\T^d)$ instead of $(x,y)\in \T^d\times\T^d$, but no better.
\end{Rem}
\smallskip
\begin{Rem}
\label{rem:necessary_marginal_entropies}
In addition to being a sufficient condition as stated above in Theorem~\ref{thm:exist_unique_MBro}, the entropy condition \eqref{eq:M_initial_final_entropy_finite} is in fact also necessary for $\MBro(\Gamma)$ to admit a (unique) solution.
Indeed, by the classical Logarithmic Sobolev Inequality, the Fisher Information controls the entropy $H(\rho)\leq C_d F(\rho)$.
Since $\FF(\PP)=\int_0^1\int F(\rho_t)\D \PP(\rho) \D t <\infty$ there exists at least a time $t_0\in[0,1]$ such that the average entropy $\int H(\rho_{t_0})\D \PP(\rho)\leq  C\int F(\rho_{t_0})\D\PP(\rho)<\infty$.
Moreover for an $AC^2$ curve the time derivative of the entropy can be computed by the chain rule $\frac{d}{dt}H(\rho_t)=\int \nabla\log\rho_t\cdot c_t\,\D\rho_t$, where $c_t(x)$ corresponds to the metric speed $\dot\rho_t$ in Theorem~\ref{thm:characterization_AC2}.
By definition of $\HH_\nu$, any plan with finite entropy $\HH_\nu(\PP)<\infty$ has both its metric speed $\dot\rho_t$ and Fisher information $F(\rho_t)$ controlled in the $L^2$ sense, hence $\frac{d}{dt} H(\rho_t)$ is controlled in $L^1$.
This $L^1$ bound allows to propagate $\int H(\rho_{t_0})\D \PP(\rho)<\infty$ to $\int H(\rho_{t})\D \PP(\rho)<\infty$ the whole interval $t\in[0,1]$, in particular to $t=0$ and $t=1$.
\end{Rem}

%%%%%%%%%%%%%%%%%%%%%%%%%%%%%%%%%%%%%%%%%%%%%%%%%%%%%%%%%%%%%%%%%%%%%%%%%%%%%%%%%%%%%%%%%%%%%%%%%%%%%%%%%%%%%%%
%%%%%%%%%%%%%%%%%%%%%%%%%%%%%%%%%%%%%%%%%%%%%%%%%%%%%%%%%%%%%%%%%%%%%%%%%%%%%%%%%%%%%%%%%%%%%%%%%%%%%%%%%%%%%%
\section{Properties of the Brownian motion on $\T^d$}
\label{section:brownian_torus}
Here we give detailed proofs of some technical results used in Section~\ref{section:GammaCv_process} for the Brownian motion and bridges on the torus, mainly Lemma~\ref{lem:estimation_int_AN_Bxy} and Lemma~\ref{lem:entropy_translation_bridges_torus}.
Throughout this appendix, barred quantities will live in $\R^d$, while unbarred quantities will live in the torus.
Typically, we shall write $\oomega\in C([0,1];\R^d)$ and $\omega\in C([0,1];\T^d)$,
and the canonical processes will read $\overline X_t:\bar\omega\mapsto\bar\omega_t$ and $X_t:\omega\mapsto\omega_t$.
The pinned Brownian motions read accordingly $\Rbar^\nu_{\bar x}$ on the whole space and $R^\nu_x$ on the torus.
%
%%%%%%%%%%%%%%%%%%%%%%%%%%%%%%%%%%%%%%%%%%%%%%%%%%%%%%%%
\subsection{Brownian bridges on the torus}
By definition, the pinned Brownian motion on $\T^d$ is nothing but the projection of a pinned Brownian motion on $\R^d$, \emph{i-e} $R^\nu_x=\Pi\pf \Rbar^\nu_{\bar x}$ as soon as the starting points are chosen consistently, $x=\pi(\bar x)$.
However, the bridges on $\T^d$
$$
R^{\nu,x,y} \coloneqq  R^\nu_x(\,\bullet\,|\,X_1=y) = R^{\nu}(\,\bullet\,|\,X_0=x,X_1=y)\phantom{.}
$$
are \emph{not} the projection of the bridges on $\R^d$
$$
\Rbar^{\nu,\bar x,\bar y} \coloneqq  \Rbar^{\nu}_{\bar x}(\,\bullet\,|\,\overline X_1=\bar y).
\phantom{= R^{\nu}(\,\bullet\,|\,X_0=x,X_1=y)}
$$
The former can however be expressed as mixtures of the latter
\begin{Lem}
	\label{lem:brownian_bridge}
	Take $x$ and $y$ in $\T^d$, and choose any lifts ${\bar x}$ and ${\bar y}$ in $\R^d$ such that $\pi({\bar x})=x$ and $\pi({\bar y})=y$.
	Then
	\begin{equation*}
	R^{\nu, x, y} = \frac{1}{Z^{\nu,x,y}}\sum_{\bar l \in \Z^d} \exp\left( - \frac{|{\bar y}-{\bar x} + \bar l|^2}{2 \nu} \right) \Pi \pf \Rbar^{\nu, {\bar x}, {\bar y} + \bar l},
	\end{equation*}
	where $Z^{\nu, x,y}\coloneqq \sum\limits_{\bar l \in \Z^d} \exp\left( - \frac{|{\bar y}-{\bar x} + \bar l|^2}{2 \nu} \right)>0$ is a normalization constant.
\end{Lem}

Remark that because of \eqref{eq:explicit_heat_kernel}, 
\begin{equation}
\label{eq:defZ}
Z^{\nu,x,y} =(2 \pi \nu)^{d/2} \tau_{\nu}(y-x).
\end{equation} 
\begin{proof}
	We only give the proof for $x=0$, the general case is identical by shift invariance.
	We first observe that the right-hand side in our statement obviously does not depend on the particular choice of the lifts, hence for $x=0$ it suffices to establish equality with $\overline{x} = 0$.
	Let $\overline{R}{}_0^\nu$ and $R_0^\nu = \Pi \pf \overline{R}{}^\nu_0$ be the Brownian motions started from the origin on $\R^d$ and $\T^d$, respectively, and consider $i: \T^d \to \R^d$ a measurable right inverse of the canonical projection $\pi: \R^d \to \T^d$.
	Applying \eqref{eq:disintegration_measure} to disintegrate $p = \overline{R}{}_0^\nu$ with respect to $\Phi = \overline X_1$, we have
	\begin{equation*}
	\overline{R}{}_0^\nu =  \int_{\R^d} \overline{R}{}^{\nu,0,\overline{y}}\frac{1}{\sqrt{2\pi\nu}^d} \exp \left( - \frac{|\overline{y}|^2}{2 \nu} \right) \D \overline{y}.  
	\end{equation*}
	Partitioning $\R^d$ in cubes $\left\{\overline{l} + i(\T^d)\right\}_{\overline{l} \in \Z^d}$ leads to
	\begin{equation*}
	\overline{R}{}_0^\nu =  \int_{\T^d} \frac{1}{\sqrt{2\pi\nu}^d}\sum_{\overline{l} \in \Z^d}\exp \left( - \frac{|i(y) + \overline{l}|^2}{2 \nu} \right) \overline{R}{}^{\nu,0,i(y) + \overline{l}}\D y,
	\end{equation*}
	hence by linearity of the pushforward operation $\Pi \pf$ we get
	\begin{equation*}
	R_0^\nu = \Pi \pf \overline{R}{}^\nu_0= \int_{\T^d}
	\Bigg\{ 
	\frac{1}{\sqrt{2\pi\nu}^d} \sum_{\overline{l} \in \Z^d}\exp \left( - \frac{|i(y) + \overline{l}|^2}{2 \nu} \right) \Pi \pf\overline{R}{}^{\nu,0,i(y) + \overline{l}}
	\Bigg\}\D y.  
	\end{equation*}
	Remark next that, for all $y$ and by definition of the inverse $\pi\circ i(y)=y$, the integrand in this right-hand side is a measure supported on the fiber $\{X_1 = y\}$: By uniqueness in the disintegration theorem \cite[Thm. 5.3.1]{ambrosio2006gradient} one simply reads off the last equality the conditioning
	$$
	 R^{\nu,0,y}=R^\nu_0(\bullet\,|\,X_1=y)
	 =\frac{1}{\sqrt{2\pi\nu}^d} \sum_{\overline{l} \in \Z^d}\exp \left( - \frac{|i(y) + \overline{l}|^2}{2 \nu} \right) \Pi \pf\overline{R}{}^{\nu,0,i(y) + \overline{l}}\\
	 $$
	 and the result follows.
\end{proof}
\qed
%
%%%%%%%%%%%%%%%%%%%%%%%%%%%%%%%%%%%%%%%%%%%%%%%%%%%%%%%
\subsection{Proof of Lemma~\ref{lem:estimation_int_AN_Bxy}}
Let us start with estimate \eqref{eq:estimation_int_AN_B}.
Since the increments of the Brownian motion are independent and stationary we have first
\begin{multline}
\label{eq:independent_increments}
 \int \exp\left(\alpha \frac{A_N(\omega)}{\nu}\right)\D R^\nu(\omega)
 =\int \exp\left(\alpha \sum\limits_{n=0}^{N-1}\frac{\dist^2(\omega_{t_n},\omega_{t_{n+1}})}{2\nu\tau}\right)\D R^{\nu}(\omega)
 \\
=  \left[\int \exp\left(\alpha \frac{\dist^2(\omega_{0},\omega_{\tau})}{2\nu\tau}\right)\D R^{\nu}(\omega)\right]^N
\end{multline}
with $\tau=1/N$.
Whence, by definition of the reversible Brownian motion $R^{\nu}$,
\begin{multline*}
\int \exp  \left(\alpha \frac{\dist^2(\omega_{0},\omega_{\tau})}{2\nu\tau}\right)\D R^{\nu}(\omega)
\\
= \frac{1}{(2\pi \nu \tau)^{d/2}} \int_{\R^d} \exp\left(\alpha \frac{\dist^2(0, \pi(\bar y))}{2\nu\tau}\right) \exp\left( -\frac{|\bar y|^2}{2\nu \tau} \right) \D \bar y
\\
\leq \frac{1}{(2\pi \nu \tau)^{d/2}} \int_{\R^d} \exp\left((\alpha-1) \frac{|\bar y|^2}{2\nu\tau}\right) \D \bar y,
\end{multline*}
where we used $\dist(0,\pi(\bar y))\leq |\bar y-0|$ in the last line.
Because we were cautious enough to choose $\alpha < 1$ this quantity is finite, and changing variables $\bar z=\sqrt{\frac{1-\alpha}{\nu \tau}} \bar y$ in the integral yields 
\begin{align}
\notag \int \exp\left(\alpha \frac{\dist^2(\omega_{0},\omega_{\tau})}{2\nu\tau}\right)\D R^{\nu}(\omega) &\leq  \frac{1}{(2\pi (1-\alpha) )^{d/2}} \int_{\R^d} \exp\left( -\frac{|\bar z|^2}{2} \right) \D \bar z \\
\label{eq:estimate_one_increment} &= \frac{1}{(1-\alpha)^{d/2}}.
\end{align}
Gathering \eqref{eq:independent_increments}\eqref{eq:estimate_one_increment} gives exactly \eqref{eq:estimation_int_AN_B}.
\\

For the conditioned version \eqref{eq:estimation_int_AN_Bxy}, choose arbitrary lifts ${\bar x},\bar y\in \R^d$ of $x,y\in\T^d$.
From Lemma~\ref{lem:brownian_bridge} we deduce 
	\begin{multline}
	\label{eq:expectation_action_bridge}
	\int  \exp\left(\frac{\alpha}{\nu} A_N(\omega)\right) \D R^{\nu,x,y}(\omega)\\
	= \frac{1}{Z^{\nu,x,y}} \hspace{-3pt}\sum_{\bar l \in \Z^d} \exp\left( - \frac{|{\bar y}-{\bar x} + \bar l|^2}{2\nu} \right)
	\hspace{-3pt}\int \hspace{-2pt} \exp\left(\frac{\alpha}{\nu} A_N\hspace{-1pt}\circ\hspace{-1pt}\Pi(\oomega)\right) \hspace{-2pt} \D \Rbar{}^{\nu,{\bar x},{\bar y}+\bar l}(\oomega).
	\end{multline}
	For arbitrary points ${\bar p},{\bar q}\in\R^d$ we first estimate the terms
	\begin{multline*}
	\Lambda({\bar p},{\bar q})\coloneqq  \int \exp\left(\frac{\alpha}{\nu} A_N\circ\Pi(\oomega))\right) \D \overline{R}{}^{\nu,{\bar p},{\bar q}}(\oomega)\\
	=
	\int \exp\left(\frac{\alpha}{\nu} \sum\limits_{n=0}^{N-1}\frac{\dist^2(\pi(\oomega_{t_n}),\pi(\oomega_{t_{n+1}}))}{2\tau}\right) \D \overline{R}{}^{\nu,{\bar p},{\bar q}}(\oomega)
	\end{multline*}
	appearing in the summand of \eqref{eq:expectation_action_bridge}.
	Since $\dist(\pi(\overline u), \pi(\overline v)) \leq |\overline v-\overline u|$ we have
	\begin{align*}
	\Lambda(\bar p,\bar q) &\leq \int \exp\left( \frac{\alpha}{2 \nu \tau} \sum_{n=0}^{N-1} |\oomega_{t_{n+1}} - \oomega_{t_n}|^2\right)\D \overline{R}{}^{\nu, \bar p, \bar q}(\oomega)\\
	&=\E_{\overline{R}{}^{\nu, \bar p,\bar q}}\left[ \exp\left( \frac{\alpha}{2 \nu \tau} \sum_{n=0}^{N-1} |\overline X_{t_{n+1}} - \overline X_{t_n}|^2\right) \right].
	\end{align*}
	Moreover, the law of the canonical process $\overline X_t$ under the bridge $\overline{R}{}^{\nu, \bar p,\bar q}$ is the same as the law of $\overline Y_t\coloneqq \overline X_t + (1-t)(\bar p-\overline X_0) + t(\bar q-\overline X_1)$ under the law of the standard Brownian motion $\Rbar^{\nu}_0$ started from the origin.
	Thus
	\begin{align*}
	\Lambda(\bar p,\bar q) 
	&  \leq 
	\E_{\overline{R}{}^{\nu}_0}\left[ \exp\left( \frac{\alpha}{2 \nu \tau} \sum_{n=0}^{N-1} |\overline Y_{t_{n+1}} - \overline Y_{t_n}|^2\right) \right]\\
	&  =
	\E_{\overline{R}{}^{\nu}_0}\left[ \exp\left( \frac{\alpha}{2 \nu \tau} \sum_{n=0}^{N-1} \Big|(\overline X_{t_{n+1}} - \overline X_{t_n}) + \tau \Big\{(\bar q - \bar p) - (\overline X_1 - \overline X_0)\Big \}\Big|^2\right) \right] .
	\end{align*}
	Expanding $(a+b-c)^2$ in the sum and recalling that $N \tau = 1$, it is easy to get:
	\begin{align*}
	&\sum_{n=0}^{N-1} \left|(\overline X_{t_{n+1}} - \overline X_{t_n}) + \tau (\bar q - \bar p) - \tau(\overline X_1 - \overline X_0)\right|^2 \\
	&\hspace{1cm}
	= \sum_{n=0}^{N-1} \left|\overline X_{t_{n+1}} - \overline X_{t_n}\right|^2 
	+ \sum_{n=0}^{N-1} \tau^2 \left |\bar q - \bar p\right |^2
	+ \sum_{n=0}^{N-1} \tau^2 \left |\overline X_1 - \overline X_0\right |^2
	\\
	&\hspace{2cm}
	+2 \sum_{n=0}^{N-1}(\overline X_{t_{n+1}} - \overline X_{t_n})\cdot\tau (\bar q - \bar p)
	-2 \sum_{n=0}^{N-1}\tau(\bar q-\bar p)  \cdot \tau (\overline X_1 - \overline X_0)\\
	& \hspace{2cm}  
	- 2 \sum_{n=0}^{N-1}(\overline X_{t_{n+1}} - \overline X_{t_n})\cdot\tau (\overline X_1 - \overline X_0)
	\\
	&\hspace{1cm}
	= \sum_{n=0}^{N-1} |\overline X_{t_{n+1}} - \overline X_{t_n}|^2 + \tau | \bar q - \bar p|^2 + \tau |\overline X_1 - \overline X_0|^2\\
	& \hspace{2cm}
	+ 2\tau (\overline X_1-\overline X_0)(\bar q-\bar p) 
	-2\tau (\bar q-\bar p)(\overline X_1-\overline X_0)
	- 2\tau |\overline X_1-\overline X_0|^2
	\\
	& \hspace{1cm}
	= \sum_{n=0}^{N-1} |\overline X_{t_{n+1}} - \overline X_{t_n}|^2 + \tau | \bar q - \bar p|^2 - \tau |\overline X_1 - \overline X_0|^2.
	\end{align*} 
	As a consequence, and by independence of the Brownian increments,
	\begin{align*}
	\Lambda &(\bar p,\bar q)
	\\
	&\leq \exp\left( \frac{\alpha}{2 \nu} |\bar q - \bar p|^2 \right)\! \E_{\overline{R}{}^{\nu}_0}\hspace{-4pt}\left[ \exp \! \left( \! \frac{\alpha}{ 2\nu \tau }\! \left\{\! \sum_{n=0}^{N-1} |\overline X_{t_{n+1}} - \overline X_{t_n}|^2 - \tau|\overline X_1 - \overline X_0|^2 \! \right\} \! \right) \! \right] 
	\\
	&\leq \exp\left( \frac{\alpha}{2 \nu} |\bar q - \bar p|^2 \right)
	\E_{\overline{R}^{\nu}_0}\left[ \exp\left( \frac{\alpha}{ 2\nu \tau }  \sum_{n=0}^{N-1} |\overline X_{t_{n+1}} - \overline X_{t_n}|^2 \right) \right] 
	\\
	&=
	\exp\left( \frac{\alpha}{2 \nu} |\bar q - \bar p|^2 \right)\times
	\left(\E_{\overline{R}^{\nu}_0}\left[ \exp\left( \frac{\alpha}{ 2\nu \tau }|\overline X_{\tau} - \overline X_{0}|^2 \right) \right]\right)^N
	\\
	&= \displaystyle{\exp\left(\frac{\alpha}{2 \nu} |\bar q - \bar p|^2 \right)}\times\frac{1}
	{(1 - \alpha)^{Nd/2}},
	\end{align*}
	where the last equality follows from the same explicit computation as in \eqref{eq:estimate_one_increment} with $\overline X_0=0$ for $\overline R^\nu_0$-almost all $\omega$.
	
	Finally, setting $\bar p=\bar x$ and $\bar q=\bar y+\bar l$ as in \eqref{eq:expectation_action_bridge}, using formulas \eqref{eq:explicit_heat_kernel}\eqref{eq:defZ} and the dimensional bounds \eqref{eq:estim_heat_flow} on the heat kernel, we get when $\nu \leq 1$:
	\begin{align*}
	\int \exp& \Big(\frac{\alpha}{\nu} A_N(\omega)\Big) \D R^{\nu,x,y}(\omega) \\
	&\leq \frac{1}{(1 - \alpha)^{dN/2}} \times\frac{1}{Z^{\nu,x,y}} \sum_{\bar l \in \Z^d} \exp\left( - (1 - \alpha) \frac{|\bar y-\bar x + \bar l|^2}{2 \nu} \right)\\
	&= \frac{1}{(1 - \alpha)^{dN/2}} \times \frac{\tau_{\frac{\nu}{1-\alpha}}(y-x)}{\tau_{\nu}(y-x)} \\
	&\leq \frac{1}{(1 - \alpha)^{dN/2}}  \times \frac{K_d\exp\left( -\frac{(1-\alpha)}{2\nu} \dist^2(x,y) \right)}{k_d\exp\left( -\frac{1}{2\nu} \dist^2(x,y) \right)}\\
	&\leq \frac{C_d}{(1 - \alpha)^{dN/2}} \exp\left( \frac{\alpha}{2\nu} \dist^2(x,y) \right)
	\end{align*}
	and the proof is achieved.
	\qed
%
%%%%%%%%%%%%%%%%%%%%%%%%%%%%%%%%%%%%%%%%%%%%%%%%%%%%%%%%	
\subsection{Proof of Lemma~\ref{lem:entropy_translation_bridges_torus}}
We first establish the corresponding result in the whole space.
As in the torus in \eqref{eq:def_translation}, we define the translation operator in $\R^d$ as:
\begin{equation*}
\Tbar_{\overline{\omega}}:\quad  \overline{\alpha} \in C([0,1]; \R^d) \quad  \mapsto \quad  \overline{\omega} + \overline{\alpha} \in C([0,1]; \R^d).
\end{equation*}
We recall that $\Bbar^{\nu} = \Rbar^{\nu, 0,0}$ is the Brownian bridge of diffusivity $\nu$ on $\R^d$ joining $0$ to $0$. We have:
\begin{Lem}
	\label{lem:entropy_translation_bridge}
	Let $\overline{\omega} \in AC^2([0,1]; \R^d)$ and $\nu >0$.
	Then 
	\begin{equation*}
	\nu H\left(\Tbar_{\overline{\omega}} {}\pf \Bbar^{\nu} \,|\, \Rbar^{\nu, \overline{\omega}_0, \overline{\omega}_1}\right) = \frac{1}{2} \int_0^1 |\dot{\overline{\omega}}_t|^2 \D t - \frac{|\overline{\omega}_1 - \overline{\omega}_0|^2}{2}.
	\end{equation*}
\end{Lem}
\begin{proof}
	We will rather establish the following equivalent formula: if ${\overline\alpha} \in AC^2([0,1]; \R^d)$ satisfies ${\overline\alpha}_0 = {\overline\alpha}_1 = 0$, then for all $\nu >0$ and $\bar x,\bar y \in \R^d$,
	\begin{equation}
	\label{eq:entropy_translated_alpha}
	\nu H\left(\Tbar_{{\overline\alpha}} {}\pf \Rbar^{\nu, \bar x, \bar y} \, \Big| \, \Rbar^{\nu, \bar x, \bar y}\right) = \frac{1}{2}\int_0^1 |\dot{{\overline\alpha}}_t|^2 \D t.
	\end{equation}
	If $\overline\xi_t\coloneqq (1-t) \overline{\omega}_0 + t \overline{\omega}_1$, it will then suffice to apply this formula with 
	${\overline\alpha}_ t\coloneqq \overline{\omega}_t - \overline\xi_t$
	and to use the identities $\overline T_{\overline{\omega}} = \overline T_{{\overline\alpha}} \circ \overline T_{\overline\xi}$ and $\overline T_{\xi} {}\pf \Bbar^{\nu} = \Rbar^{\nu, \overline{\omega}_0, \overline{\omega}_1}$.
	
	So let us prove \eqref{eq:entropy_translated_alpha}.
	We fix ${\overline\alpha} \in AC^2([0,1]; \R^d)$ with ${\overline\alpha}_0 = {\overline\alpha}_1 = 0$ and $\nu >0$.
	First, by the standard Cameron-Martin formula, if $\overline{R}{}^{\nu}$ is any $\nu$ Brownian motion on $\R^d$ then
	\begin{equation*}
	\nu H(\Tbar_{{\overline\alpha}} {}\pf \overline{R}{}^{\nu} \, | \, \overline{R}{}^{\nu}) = \frac{1}{2} \int_0^1 |\dot{{\overline\alpha}}_t|^2 \D t.
	\end{equation*}
	Noticing that the marginals $(\overline{R}{}^{\nu})_{0,1}$ and $(T_{{\overline\alpha}} {}\pf \overline{R}{}^{\nu})_{0,1}$ coincide (because ${\overline\alpha}_0={\overline\alpha}_1=0$), we can apply Proposition~\ref{prop:disintegration_entropy} in order to condition on the endpoints $(X_0, X_1)$ and get:
	\begin{equation*}
	H(\Tbar_{{\overline\alpha}} {}\pf \overline{R}{}^{\nu} \, | \, \overline{R}{}^{\nu}) = 
	0 + \int H\Big( (\Tbar_{{\overline\alpha}} {}\pf \overline{R}{}^{\nu})^{\bar x, \bar y} \, \Big| \, \Rbar^{\nu,\bar x, \bar y}  \Big) \D \overline{R}{}^{\nu}_{0,1}(\bar x, \bar y).
	\end{equation*}
	Gathering these two formulas and observing that $(\Tbar_{{\overline\alpha}} {}\pf \overline{R}{}^{\nu})^{\bar x, \bar y} = \Tbar_{{\overline\alpha}} {}\pf (\Rbar^{\nu,\bar x, \bar y})$ we get
	\begin{equation}
	\label{eq:conditioned_CM}
	 \nu \int H( \Tbar_{{\overline\alpha}} {}\pf \Rbar^{\nu,\bar x,\bar y} \, | \, \Rbar^{\nu, \bar x,\bar y}  ) \D \overline{R}{}^{\nu}_{0,1}(x,y) = \frac{1}{2} \int_0^1 |\dot{{\overline\alpha}}_t|^2 \D t.
	\end{equation}
	Finally, take $\bar x,\bar y\in\R^d$ and consider the geodesic $\overline\xi_t\coloneqq (1-t) \bar x + t \bar y$.
	Then $\Rbar^{\nu, \bar x,\bar y} = \Tbar_{\overline\xi} {}\pf \Bbar^{\nu}$, the translations $\Tbar_{\overline\xi}$ and $\Tbar_{{\overline\alpha}}$ commute, and $\Tbar_{\overline\xi}$ is invertible.
	Whence by Proposition~\ref{prop:entropy_one_to_one}
	\begin{align*}
	H( \Tbar_{{\overline\alpha}} {}\pf \Rbar^{\nu,\bar x, \bar y} \, | \, \Rbar^{\nu, \bar x,\bar y}  ) &= H( \Tbar_{{\overline\alpha}} {}\pf \Tbar_{\overline\xi} {}\pf \Bbar^{\nu} \, | \, \Tbar_{\overline\xi} {\pf} \Bbar^{\nu}  ) \\
	&=H( \Tbar_{\xi} {}\pf \Tbar_{{\overline\alpha}} {}\pf \Bbar^{\nu} \, | \, \Tbar_{\overline\xi} {\pf} \Bbar^{\nu}  ) =H( \Tbar_{{\overline\alpha}} {}\pf \Bbar^{\nu} \, | \, \Bbar^{\nu}  ).
	\end{align*}
	Finally exploiting \eqref{eq:conditioned_CM}, we get for all $\bar x$ and $\bar y$ in $\R^d$:
	\begin{equation*}
	\nu H( \Tbar_{{\overline\alpha}} {}\pf \Rbar^{\nu,\bar x,\bar y} | \Rbar^{\nu, \bar x,\bar y}  ) = \nu H( \Tbar_{{\overline\alpha}} {}\pf \Bbar^{\nu} | \Bbar^{\nu}  ) =  \frac{1}{2} \int_0^1 |\dot{{\overline\alpha}}_t|^2 \D t
	\end{equation*}
	and the proof is complete.
	\qed
	\end{proof}
In order to deduce Lemma~\ref{lem:entropy_translation_bridges_torus} from Lemma~\ref{lem:entropy_translation_bridge} we need a canonical construction of a process on $\R^d$ out of a process on $\T^d$.
To this end, we choose $i : \T^d \to \R^d$ a measurable right inverse of the projection $\pi$ with bounded image.
For $\omega \in C([0,1]; \T^d)$ we denote by $I(\omega)$ the unique lift of $\omega$ starting from $i(\omega_0)$, and $I$ is of course a measurable right inverse of $\Pi$.
The entropy is invariant under the canonical projection in the following sense
\begin{Lem}
	\label{lem:Pi_one_to_one}
	Take $\overline{P}$ a probability measure and $\overline{R}$ a finite positive Radon measure on $C([0,1]; \R^d)$.
	Suppose that $\overline{P} \ll \overline{R}$ and that, $\overline{R}$-almost surely, $\overline X_0 = i(\pi(\overline X_0))$. 
	Then
	\begin{equation*}
	H(\Pi\pf \overline{P}\,|\, \Pi\pf \overline{R}) = H(\overline{P}\,|\,\overline{R}).
	\end{equation*}
\end{Lem}
\begin{proof}
On the set $\{ \overline X_0 = i(\pi(\overline X_0)) \}$ we have $\overline{R}$-almost surely $I \circ \Pi = \Id$, and our statement is a direct consequence of Proposition~\ref{prop:entropy_one_to_one}.
\qed
\end{proof}
With the above definition of the lift $I(\omega)$, observe that for all $\omega \in C([0,1]; \T^d)$ the shifted bridge $B^{\nu}_{\omega}$ from Definition~\ref{def:B^nu_omega} satisfies:
	\begin{equation*}
B^{\nu}_{\omega} \coloneqq  T_{\omega} {}\pf \Pi\pf \Bbar^{\nu} = \Pi\pf \Tbar_{I(\omega)} {}\pf \Bbar^{\nu}.
\end{equation*}

We are finally in position of establishing Lemma~\ref{lem:entropy_translation_bridges_torus}.
\begin{proof}[Proof of Lemma~\ref{lem:entropy_translation_bridges_torus}]
For notational convenience, we denote the lift of $\omega$ by
	\begin{equation*}
	\oomega\coloneqq I(\omega)\in C([0,1];\R^d).
	\end{equation*}
	By Lemma~\ref{lem:brownian_bridge}, we have
	\begin{align*}
	R^{\nu, \omega_0, \omega_1} &= \frac{1}{Z^{\nu,\omega_0,\omega_1}} \sum_{\bar l \in \Z^d} \exp\left( - \frac{|\oomega_1-\oomega_0 + \bar l|^2}{2 \nu} \right) \Pi \pf \Rbar^{\nu, \oomega_0, \oomega_1 + \bar l}\\
	&= \Pi\pf 
	\Bigg(
	\underbrace{
	\frac{1}{Z^{\nu,\omega_0,\omega_1}} \sum_{\bar l \in \Z^d} \exp\left( - \frac{|\oomega_1-\oomega_0 + \bar l|^2}{2 \nu} \right) \Rbar^{\nu, \oomega_0, \oomega_1 + \bar l}}_{ \coloneqq  \overline{B} {}^{\nu, \omega_0, \omega_1}}
	\Bigg). 
	\end{align*}
	(The superscripts in $\overline{B} {}^{\nu, \omega_0, \omega_1}$ do \emph{not} stand for the conditioning of some $\overline B^\nu$, but rather emphasizes the dependence of the measure $\overline{B} {}^{\nu, \omega_0, \omega_1}$ on the fixed endpoints $\omega_0,\omega_1\in \R^d$.)
	Observe that 
	\begin{itemize}
	\item
	all the measures involved in the definition of $\overline{B} {}^{\nu, \omega_0, \omega_1}$ are mutually singular (because $\Rbar^{\nu,\overline a,\overline b}\perp \Rbar^{\nu,\overline a,\overline b'}$ as soon as $\overline b\neq \overline b'$),
	\item 
	$\Tbar_{\oomega} {}\pf \Bbar^{\nu} \ll \Rbar^{\nu, \oomega_0, \oomega_1}$ by Lemma~\ref{lem:entropy_translation_bridge},
	\item
	$  \Rbar^{\nu, \oomega_0, \oomega_1} \ll  \overline{B} {}^{\nu, \omega_0, \omega_1}$ (because $\Rbar^{\nu, \oomega_0, \oomega_1}$ appears in the sum defining the measure $\overline{B} {}^{\nu, \omega_0, \omega_1}$ for $\bar l=0$),
	\end{itemize}
	As a consequence $\Tbar_{\oomega} {}\pf \Bbar^{\nu} \ll \overline{B} {}^{\nu, \omega_0, \omega_1}$, and the corresponding Radon-Nikodym derivative only involves the $\bar l=0$ contribution in $\overline B {}^{\nu,\omega_0,\omega_1}$.
	
	Moreover by definition~\ref{def:B^nu_omega} of $B^\nu_\omega=T_\omega {}\pf \Pi {}\pf \overline B ^\nu$ and because $T_\omega\circ \Pi=\Pi\circ \Tbar_{\oomega}$ one can write $B^\nu=\Pi{}\pf\Tbar_{\oomega} {}\pf \Bbar^{\nu}$.
	Since $\overline{B} {}^{\nu, \omega_0, \omega_1}$-almost surely $\overline X_0 = i(\pi(\overline X_0))$,
	we can apply Lemma~\ref{lem:Pi_one_to_one}
	\begin{align*}
	H(B^{\nu}_{\omega}\, | \, R^{\nu, \omega_0, \omega_1}) &=  
	H( \Pi{\pf}\Tbar_{\oomega} {}\pf \Bbar^{\nu}\, | \,\Pi{}\pf\overline{B} {}^{\nu, \omega_0, \omega_1})
	\\
	& = H( \Tbar_{\oomega} {}\pf \Bbar^{\nu}\, | \,\overline{B} {}^{\nu, \omega_0, \omega_1})
	\\
	&=H\left(  \Tbar_{\oomega} {}\pf \Bbar^{\nu} \,\Bigg | \, \frac{1}{Z^{\nu,\omega_0,\omega_1}} \exp\left(- \frac{| \oomega_1-\oomega_0 |^2}{2\nu}\right)  \Rbar^{\nu, \oomega_0, \oomega_1} \right)
	\\
	&= H\left(  \Tbar_{\oomega} {}\pf \Bbar^{\nu} \, | \, \Rbar^{\nu, \oomega_0, \oomega_1}\right)
	+ \frac{| \oomega_1-\oomega_0 | ^2}{2\nu} + \log Z^{\nu,\omega_0,\omega_1}
	\end{align*}
	where we used Proposition~\ref{prop:entropy_wrt_fR} in the last equality.
	We can compute the first entropy term in the right hand side using Lemma~\ref{lem:entropy_translation_bridge} (remark that the kinetic action of $\omega$ on the torus coincides with that of its lift $\oomega$) and we can estimate the last term using \eqref{eq:defZ}\eqref{eq:estim_heat_flow}, which leads to
	\begin{align*}
	\nu H(B^{\nu}_{\omega}\,|\, R^{\nu, \omega_0, \omega_1}) &\leq 
	\left(\frac 12\int_0^1|\dot\oomega_t|^2\,\D t  -  \frac{1}{2}|\oomega_1-\oomega_0|^2\right)\\
	&\hspace{2cm} + \nu\left(\frac{|\oomega_1-\oomega_0|^2}{2\nu} + \log K_d - \frac{\dist^2(\omega_0, \omega_1)}{2 \nu}\right)
	\\
	& =\frac{1}{2} \int_0^1 |\dot{\omega}_t|^2 \D t  - \frac{\dist^2(\omega_0, \omega_1)}{2 \nu} + \nu \log K_d
	\end{align*}
	and concludes the proof with $C \coloneqq  \log K_d$.
	\qed
\end{proof}

\end{appendices}

% \end{acknowledgements}

\bibliography{bibliography}
\bibliographystyle{plain}

% \begin{thebibliography}{}
% Author, Article title, Journal, Volume, page numbers (year)
% Author, Book title, page numbers. Publisher, place (year)
% \end{thebibliography}

\end{document}